\documentclass{article}

\usepackage{fancyhdr}
{}
\usepackage[utf8]{inputenc}
\usepackage[top=1.15in, bottom=1.15in, left=1.17in, right=1.17in]{geometry}
\usepackage{amsthm,dsfont, amsmath, amssymb, amscd, latexsym, multicol, verbatim, enumerate, graphicx,xy, color}
\usepackage{amstext,amsfonts,amssymb,amscd,amsbsy,amsmath}
\usepackage{authblk}
\usepackage[nottoc,numbib]{tocbibind}
\usepackage[titletoc]{appendix}
\usepackage{tikz,float}
\usetikzlibrary{patterns}
\usepackage[english]{babel}
\usepackage{cite}
\usepackage{stmaryrd}
\usepackage[colorlinks=true, pdfstartview=FitV, linkcolor=blue, citecolor=blue, urlcolor=blue, breaklinks=true]{hyperref}

\usepackage{helvet}         
\usepackage{courier}        
\usepackage{type1cm}  
\usepackage{multicol}        
\usepackage[bottom]{footmisc}
\usepackage{longtable}

\usepackage{hyperref}
\usepackage{makeidx}         
\usepackage{graphicx}        

\makeindex             

\usepackage{amstext,amsfonts,amssymb,amscd,amsbsy,amsmath}
\usepackage[ruled,vlined]{algorithm2e}
\usepackage{algpseudocode}
\usepackage{tikz-cd}	
\usepackage{tikz}
\usetikzlibrary{arrows}

\usetikzlibrary{decorations.markings}
\usepackage{youngtab}

\numberwithin{equation}{section}

\theoremstyle{definition}
\newtheorem{definition}{Definition}
\newtheorem{proposition}{Proposition}
\newtheorem{theorem}{Theorem}
\newtheorem{lemma}{Lemma}
\newtheorem{corollary}{Corollary}
\newtheorem{remark}{Remark}
\newtheorem{example}{Example}

\newtheorem{conjecture}{Conjecture}
\newtheorem*{corollary*}{Corollary}
\newtheorem*{lemma*}{Lemma}
\newtheorem*{theorem*}{Theorem}
\newtheorem*{proposition*}{Proposition}
\newtheorem*{problem*}{Problem}

\newcommand{\w}{\underline{w}}
\newcommand{\p}{\mathbf{p}}

\newcommand{\boundellipse}[3]
{(#1) ellipse (#2 and #3)
}

\newcommand{\G}{\mathcal{G}}

\newcommand{\ff}{\mathcal{F}}

\newcommand{\bm}{\mathbf{m}}
\newcommand{\bn}{\mathbf{n}}
\newcommand{\bfff}{\mathbf{f}}

\newcommand{\bw}{\mathbf{w}}

\newcommand{\bu}{\mathbf{u}}

\DeclareMathOperator{\trop}{trop}

\DeclareMathOperator{\Proj}{Proj}

\newcommand{\Spec}{\operatorname*{Spec}}

\newcommand{\gr}{{\operatorname*{gr}}}
\newcommand{\lie}{\mathfrak}
\newcommand{\Gr}{{\operatorname*{Gr}}}

\newcommand{\val}{\mathfrak{v}}

\DeclareMathOperator{\Flag}{\mathcal{F}\hspace{-1.6pt}\ell}

\DeclareMathOperator{\init}{in}
\DeclareMathOperator{\wei}{wt}
\DeclareMathOperator{\sgn}{sgn}

\DeclareMathOperator{\conv}{conv}
\DeclareMathOperator{\cone}{cone}
\DeclareMathOperator{\rank}{rank}

\usepackage[utf8]{inputenc}

\title{Full-rank Valuations and Toric Initial Ideals}
\author{Lara Bossinger\footnote{Supported by "Programa de Becas Posdoctorales en la UNAM 2018" Instituto de Matem\'aticas, UNAM, and Max Planck Institute for Mathematics in the Sciences, Leipzig.}}
\date{}

\begin{document}

\maketitle
\vspace{-.75cm}
\begin{abstract}
Let $V(I)$ be a polarized projective variety or a subvariety of a product of projective spaces and let $A$ be its (multi-)homogeneous coordinate ring. To a full-rank valuation $\val$ on $A$ we associate a weight vector $w_\val$. Our main result is that the value semi-group of $\val$ is generated by the images of the generators of $A$ if and only if the initial ideal of $I$ with respect to $w_\val$ is prime. As application we prove a conjecture by \cite{BLMM} connecting the Minkowski property of string polytopes to the tropical flag variety.
For Rietsch-Williams' valuation for Grassmannians we identify a class of plabic graphs with non-integral associated Newton--Okounkov polytope (for $\Gr_k(\mathbb C^n)$ with $n\ge 6$ and $k\ge 3$).
\end{abstract}

\section{Introduction}

In the context of toric degenerations\footnote{A \emph{toric degeneration} of a projective variety $X$ is a flat morphism $\pi:\mathcal X\to \mathbb A^m$ with generic fiber $\pi^{-1}(t)$ for $t\not=0$ isomorphic to $X$ and $\pi^{-1}(0)$ a projective toric variety.} of projective varieties the study of full-rank valuations on homogeneous coordinate rings  (see Definition~\ref{def: valuation}) is very popular.
This goes back to a result of Anderson \cite{An13}: if the semi-group, which is the image of the valuation (called \emph{value semi-group}) is finitely generated, the valuation defines a toric degeneration.
Therefore, a hard and central question is whether a given valuation has finitely generated value semi-group or not.

For example, the \emph{valuations from birational sequences} in \cite{FFL15} are of full rank, and constructed to define toric degenerations of flag and spherical varieties. 
But for a general valuation arising in this setting it remains unknown if its value semi-group is finitely generated.
In the recent paper \cite{B-birat} a new class of valuations from birational sequences for Grassmannians is constructed.
The author applies results of this paper to identify those giving toric degenerations.

To gain more control over the valuation it is moreover desirable to identify algebra generators of $A$ whose valuation images generate the value semi-group. 
Such generators are called a \emph{Khovanskii basis}, introduced by Kaveh-Manon in \cite{KM16}. 
They construct full rank valuations with finite Khovanskii bases from maximal prime cones of the tropicalization of a polarized projective variety. 
In this paper, we complement their work by taking the opposite approach: starting from a full rank valuation we give a criterion for when it comes from a maximal prime cone.
Our main tool is reformulating the problem in terms of initial ideals in Gr\"obner theory.

\medskip

Throughout the paper for $n\in \mathbb Z_{>0}$ let $[n]$ denote $\{1,\dots,n\}$.
Let $X$ be a subvariety of the product of projective spaces $\mathbb P^{k_1-1}\times \dots \times \mathbb P^{k_s-1}$. In particular, if $s=1$ then $X$ is a polarized projective variety.
Its (multi-)homogeneous coordinate ring $A$ is given by $\mathbb C[x_{ij}\vert i\in[s],j\in[k_i]]/I$.
Here $I$ is a prime ideal in $S:=\mathbb C[x_{ij}\vert i\in[s],j\in[k_i]]$, the total coordinate ring of $\mathbb P^{k_1-1}\times \dots \times \mathbb P^{k_s-1}$. 
Further, $I$ is homogeneous with respect to the $\mathbb Z_{\ge 0}^s$-grading on $S$.
By $\bar x_{ij}\in A$ we denote the cosets of variables $x_{ij}$.
Let $d$ be the Krull-dimension of $A$.

A valuation $\val:A\setminus \{0\}\to \mathbb Z^d$ has \emph{full rank} if its image (the value semi-group $S(A,\val)$) spans a sublattice of full rank in $\mathbb Z^d$.
It is \emph{homogeneous} if it respects the grading on $A$.
From now on we only consider full-rank valuations.
In Definition~\ref{def: wt matrix from valuation} we define the \emph{weighting matrix} of $\val$ as $M_\val:=(\val(\bar x_{ij}))_{ij}\in \mathbb Z^{d\times (k_1+\dots +k_s)}$.

By means of higher Gr\"obner theory (see for example, \cite[\S8]{KM16}), we consider the \emph{initial ideal} $\init_{M_\val}(I)\subset S$ of $I$ with respect to $M_\val$ (see Definition~\ref{def: init wrt M}).
Our main result is the following theorem. It is formulated in greater detail in Theorem~\ref{thm: val and quasi val with wt matrix} below.

\begin{theorem*}
Let $\val:A\setminus\{0\}\to \mathbb Z^d$ be a full-rank valuation 
with full rank weighting matrix $M_\val\in\mathbb Z^{d\times (k_1+\dots+k_s)}$ for the presentation $S/I$ of $A$. Then
\begin{center}
 $S(A,\val)$ is generated by $\{\val(\bar x_{ij})\}_{i\in[s],j\in[k_i]}\quad \quad \text{ if and only if} \quad \quad \init_{M_\val}(I)$ is prime. 
\end{center}
\end{theorem*}

The theorem has some very interesting implications in view of toric degenerations and Newton--Okounkov bodies.
Consider a valuation of form $\val:A\setminus\{0\}\to \mathbb Z_{\ge 0}^{s}\times \mathbb Z^{d-s}$ with $\val(f)=(\deg f,\cdot)$ for all $f\in A$.
Without loss of generality by \cite[Remark 2.6]{IW18} we may assume that any full-rank homogeneous valuation is of this form.
The Newton--Okounkov cone $C(A, \val)\subset \mathbb R^{s}\times \mathbb R^{d-s}$ is the cone over its image.
The \emph{Newton--Okounkov body} $\Delta(A,\val)$ is then the intersection of $C(A,\val)$ with the subspace $\{(1,\dots,1)\}\times \mathbb R^{d-s}$, for details consider Definition~\ref{def: val sg NO}.
Newton--Okounkov bodies are in general pretty wild objects, they are convex bodies but need neither be polyhedral nor finite.
However, if $S(A,\val)$ is finitely generated they are rational polytopes (this case for $s=1$ is treated by Anderson).

Recall, that if $S(A,\val)$ is generated by $\{\val(\bar x_{ij})\}_{i\in[s],j\in[k_i]}$ (i.e. $\{\bar x_{ij}\}_{i\in [s],j\in [k_i]}$ is a \emph{Khovanskii basis} for $(A,\val)$) this has a number of useful consequences.
For example, in this case the associated graded algebra of $\val$ can be presented as $S/\init_{M_\val}(I)$.
The following corollary is another such consequence and crucial for our application to flag varieties:

\begin{corollary*}
Let $\val:A\setminus\{0\}\to \mathbb Z^d$ be a full-rank valuation 
with $M_\val\in\mathbb Z^{d\times (k_1+\dots+k_s)}$ the weighting matrix of $\val$ for the presentation $S/I$ of $A$. Assume additionally $\init_{M_\val}(I)$ is prime, hence $S(A,\val)$ is generated by $\{\val(\bar x_{ij})\}_{i\in[s],j\in[k_i]}$. Then
 the Newton--Okounkov polytope is the Minkowski sum\footnote{For two polytopes $A,B\subset \mathbb R^d$ their \emph{Minkowski sum} is defined as $A+B:=\{a+b\mid  a\in A,b\in B\}\subset \mathbb R^d.$}: 
\[
\Delta(A,\val)= \conv(\val(\bar x_{1j}))_{ j\in[k_1]}+\dots+ \conv(\val(\bar x_{sj}))_{j\in[k_s]}. 
\]
\end{corollary*}

We continue our study by considering monomial maps as appear for example in \cite{MoSh}.
Let $\phi_\val:S\to \mathbb C[y_1,\dots,y_d]$ be the homomophism defined by sending a generator $x_{ij}$ to the monomial in $y_k$'s with exponent vector $\val(\bar x_{ij})$.
Its kernel $\ker(\phi_\val)\subset S$ is a toric ideal, see \eqref{eq:kernel phi and gr}.

Further analyzing our weighting matrices, we associate to each a \emph{weight vector}. 
Let $w_\val\in \mathbb Z^{k_1+\dots+k_s}$ be a weight vector associated to $M_\val$ satisfying $\init_{M_\val}(I)=\init_{w_\val}(I)$ (see Lemma~\ref{lem: wt for matrix}).
The following lemma reveals the relation between $w_\val$ and the toric ideal $\ker(\phi_\val)$, see also Lemma~\ref{lem: wt matrix val in trop}.

\begin{lemma*}
For every full-rank valuation  $\val:A\setminus\{0\}\to \mathbb Z^d$ we have  $\init_{w_\val}(I)\subset \ker(\phi_\val)$.
In particular, $\init_{w_\val}(I)$ is monomial-free and $w_\val$ is contained in the tropicalization of $I$ (in the sense of \cite{M-S}, see \eqref{eq:trop}).
\end{lemma*}

We apply our results to two classes of valuations.

\paragraph{Grassmannians and valuations from plabic graphs.}
We consider a class of valuations on the homogeneous coordinate rings of Grassmannians defined in \cite{RW17}.
In the context of cluster algebras and cluster duality for Grassmannians, they associate a full-rank valuation $\val_\G$ to every plabic graph $\mathcal G$ with certain properties \cite{Pos06}.

For $k<n$ denote by $\binom{[n]}{k}$ the set of $k$-element subsets of $[n]$. 
Consider the Grassmannian with its Pl\"ucker embedding $\Gr_k(\mathbb C^n)\hookrightarrow \mathbb P^{\binom{n}{k}-1}$.
We obtain its homogeneous coordinate ring $A_{k,n}:=\mathbb C[p_J\vert J\in \binom{[n]}{k}]/I_{k,n}$.
In particular, $I_{k,n}$ is the Pl\"ucker ideal defining the Grassmannian.
The elements $\bar p_J\in A_{k,n}$ are Pl\"ucker coordinates.
Applying our main theorem we identify a class of plabic graphs for $\Gr_k(\mathbb C^n)$ with $k\ge 3$ and $n\ge 6$ for which the Newton--Okounkov body $\Delta(A_{k,n},\val_{\G})$ is non-integral (see Theorem~\ref{thm:hexa not prime}).

Related to \cite{RW17}, in \cite{BFFHL} the authors associate \emph{plabic weight vectors} ${\bf w}_\G$ for $\mathbb C[p_J]_J$ to the same plabic graphs $\G$ (see Definition~\ref{def: plabic deg}). 
They study these weight vector for $\Gr_2(\mathbb C^n)$ and $\Gr_3(\mathbb C^6)$ and show that, in these cases, they lie in the \emph{tropical Grassmannian} \cite{SS04}, i.e. the tropicalization of $I_{k,n}$.
With our methods, we show that for $\Gr_2({\mathbb C^n})$ their weight yields the same toric degeneration as Rietsch--Williams' Newton--Okounkov polytope (see Proposition~\ref{prop: plabic lin form} for details):

\begin{proposition*}
For every plabic weight vector $\bw_\G$ for $\Gr_{2}(\mathbb C^n)$ we have $
\init_{M_\G}(I_{2,n})=\init_{\bw_\G}(I_{2,n})$.
In particular, for the associated graded of $\val_{\G}$ we have $\gr_{\val_{\G}}(A_{2,n})=\mathbb C[p_{ij}]_{ij}/\init_{\bw_\G}(I_{2,n})$.

\end{proposition*}

For more general Grassmannians Theorem~\ref{thm:hexa not prime} below shows that this is not always the case.

\paragraph{Flag varieties and string valuations.}
We consider string valuations \cite{Kav15,FFL15} on the homogeneous coordinate ring of the full flag variety $\Flag_n$\footnote{We consider the full flag variety of type $\mathtt A$, i.e. $\Flag_n:=\{ \{0\}\subset V_1\subset \dots\subset V_{n-1}\subset \mathbb C^n\mid \dim V_i=i\}$. It can also be realized as $SL_n/B$, where $B$ are upper triangular matrices in $SL_n$.}. 
They were defined to realize string parametrizations \cite{Lit98,BZ01} of Lusztig's dual canonical basis in terms of Newton--Okounkov cones and polytopes.

Consider the algebraic group $SL_n$ and its Lie algebra $\lie{sl}_n$ over $\mathbb C$.   
Fix a Cartan decomposition and take $\Lambda\cong \mathbb Z^{n-1}$ to be the weight lattice.
It has a basis of fundamental weights $\omega_1,\dots,\omega_{n-1}$ and every dominant integral weight, i.e. $\lambda \in \mathbb Z^{n-1}_{\ge0}$ yields an irreducible highest weight representation $V(\lambda)$ of $\lie{sl}_n$.
The Weyl group of $\lie{sl}_n$ is the symmetric group $S_n$. 
By $w_0\in S_n$ we denote its longest element.

For every reduced expression $\w_0$ of $w_0\in S_n$ and every dominant integral weight $\lambda \in \mathbb Z^{n-1}_{\ge0}$, there exists a \emph{string polytope} $Q_{\w_0}(\lambda)\subset \mathbb R^{\frac{n(n-1)}{2}}$.
Its lattice points parametrize a basis for $V(\lambda)$.
The string polytope for the weight $\rho=\omega_1+\dots +\omega_{n-1}$ is the Newton--Okounkov polytope for the string valuation $\val_{\w_0}$ on the homogeneous coordinate ring of $\Flag_n$.

We embed $\Flag_n$ into a product of projective spaces as follows: first, consider the embedding into the product of Grassmannians $\Gr_1(\mathbb C^n)\times\dots\times\Gr_{n-1}(\mathbb C^n)$.
Then concatenate with the Pl\"ucker embeddings $\Gr_k(\mathbb C^n)\hookrightarrow \mathbb P^{\binom{n}{k}-1}$ for every $1\le k\le n-1$.
This yields the (multi-)homogeneous coordinate ring $A_n$ of $\Flag_n$ as $\mathbb C[p_J\vert J\subset [n]]/I_n$.
Our main result applied to string valuations yields the following, for more details see Theorem~\ref{thm: quasival for string}. 
 
\begin{theorem*}
Let $\w_0$ be a reduced expression of $w_0\in S_n$ and consider the string valuation $\val_{\w_0}$.
If $\init_{M_{\val_{\w_0}}}(I_n)$ is prime, then 
\[
Q_{\w_0}(\rho)= \conv (\val_{\w_0}(\bar p_J))_{J\in\binom{[n]}{1} } + \dots + \conv(\val_{\w_0}(\bar p_J))_{ J\in\binom{[n]}{n-1}}.
\]
\end{theorem*}

A central question concerning string polytopes is the following: fix a reduced decomposition $\w_0$ and let $\lambda=a_1\omega_1+\dots a_{n-1}\omega_{n-1}$ with $a_i\in\mathbb Z_{\ge 0}$. 
\emph{Is the string polytope $Q_{\w_0}(\lambda)$ equal to the Minkowski sum $a_1Q_{\w_0}(\omega_1)+\dots +a_{n-1} Q_{\w_0}(\omega_{n-1})$ of fundamental string polytopes?}

If equality holds for all $\lambda$, we say $\w_0$ has the \emph{Minkowski property}.
In \cite{BLMM} the authors define weight vectors ${\bf w}_{\w_0}$ for every $\w_0$.
They conjecture a relation between the Minkowski property of $\w_0$ and the weight vector ${\bf w}_{\w_0}$ lying in a maximal prime cone of the \emph{tropical flag variety}, i.e. the tropicalization of $I_n$.
A corollary of our main theorem proves an even stronger version of their conjecture. 
It can be summarized as follows (for details see Corollary~\ref{cor: prime implies MP}):

\begin{corollary*}
Let $\w_0$ be a reduced expression of $w_0\in S_n$ and consider the weight vector ${\bf w}_{\w_0}$. 
Then
$\init_{{\bf w}_{{\w_0}}}(I_n)$ is prime if and only if $\w_0$  has the Minkowski property.
\end{corollary*}

The paper is structured as follows.
We recall preliminaries on valuations, toric degenerations and Newton--Okounkov bodies in \S\ref{sec:notation}.
We then turn to quasi-valuations and weighting matrices in \S\ref{sec:val and quasival} and prove our main result Theorem~\ref{thm: val and quasi val with wt matrix}.
We make the connection to weight vectors and tropicalization and prove the above mentioned result Lemma~\ref{lem: wt matrix val in trop}.
In \S\ref{sec:exp trop flag} we apply our results to string valuations for flag varieties and in \S\ref{sec: exp plabic} to valuations from plabic graphs for Grassmannians. 
The Appendix contains background information for \S\ref{sec: exp plabic}.

\medskip

{\bf Acknowledgements.} The results of this paper were mostly obtained during my PhD at the University of Cologne under Peter Littelmann's supervision. 
I am deeply grateful for his support and guidance throughout my PhD.
I would like to further thank Xin Fang, Kiumars Kaveh, Fatemeh Mohammadi and Bea Schumann for inspiring discussions.
Moreover, I am grateful to Alfredo N\'ajera Ch\'avez for helpful comments during the preparation of this manuscript and to two anonymous referees who helped improve and clarify the statements.

\section{Notation}\label{sec:notation}

We recall basic notions on valuations and Newton--Okounkov polytopes as presented in \cite{KK12}.
Let $A$ be the homogeneous coordinate ring of a projective variety or more generally the (multi-)homogeneous coordinate ring of a subvariety of a product of projective spaces, which can be described as below.

The total coordinate ring of $\mathbb P^{k_1-1}\times \dots\times\mathbb P^{k_s-1}$ for $k_1,\dots,k_s\ge 1$ is $S:=\mathbb C[x_{ij}\vert i\in[s],j\in[k_i]]$. 
It is graded by $\mathbb Z^s_{\ge 0}$ as follows.
Let $\{\epsilon_i\}_{i\in[s]}$ denote the standard basis of $\mathbb Z^s$.
The degree of coordinates is given by $\deg x_{ij}:= \epsilon_i\in\mathbb Z^s$ (see e.g. \cite[Example~5.2.2]{CLS11}) for all $i\in[s],j\in[k_i]$.
For $u\in\mathbb Z^{k_1+\dots+k_s}$ let $x^u$ denote the monomial $\prod_{i\in[s],j\in[k_i]} x_{ij}^{u_{ij}}\in S$.
We fix the lexicographic order on $\mathbb Z^s$ and consider $f=\sum a_{u}x^{u} \in S$.
Then 
\[
\deg f:=\max{}_{\text{lex}}\{\deg x^{u}\mid a_{u}\not=0\}.
\]
Let $X$ be subvariety of $\mathbb P^{k_1-1}\times \dots\times\mathbb P^{k_s-1}$ and $A$ its homogeneous coordinate ring of Krull-dimension $d$.
Then $A=S/I$ for some prime ideal $I\subset S$ that is homogeneous with respect to the $\mathbb Z^s_{\ge 0}$-grading.
The $\mathbb Z_{\ge 0}^s$-grading on $S$ induces a $\mathbb Z^s_{\ge 0}$-grading on $A$, which we denote $A=\bigoplus_{m\in\mathbb Z^s_{\ge 0}}A_m$. 
We call a grading of this form a \emph{positive (multi-)grading}.

To define a valuation on $A$ we fix a linear order $\prec$ on the additive abelian group $\mathbb Z^d$. 

\begin{definition}\label{def: valuation}
A map $\val: A\setminus\{0\}\to (\mathbb Z^d,\prec)$ is a \emph{valuation}\footnote{We assume the image of $\val$ lies in $\mathbb Z^d$ for simplicity. Without any more effort, we could replace it by $\mathbb Q^d$. The same is true for all (quasi-)valuations considered in the rest of the paper.}, if it satisfies for $f,g\in A\setminus\{0\},c\in \mathbb C^*$ 

(i) \  $\val(f+g)\succeq \min \{\val(f),\val(g)\}$,

(ii) $\val(fg)=\val(f)+\val(g)$, and  

(iii) $\val(cf)=\val(f)$.
\end{definition}
If we replace (ii) by $\val(fg)\succeq \val(f)+\val(g)$ then $\val$ is called a \emph{quasi-valuation} (also called \emph{loose valuation} in \cite{T03}).
Let $\val:A\setminus\{0\}\to(\mathbb Z^d,\prec)$ be a valuation. 
The image $\{\val(f)\mid f\in A\setminus\{0\}\}\subset \mathbb Z^d$ forms an additive semi-group. 
We denote it by $S(A,\val)$ and refer to it as the \emph{value semi-group}. 
The \emph{rank} of the valuation is the rank of the sublattice generated by $S(A,\val)$ in $\mathbb Z^d$.
We are interested in valuations of \emph{full rank}, i.e. $\rank(\val)=d$.

One naturally defines a $\mathbb Z^d$-filtration on $A$ by $F_{\val \succeq a}:=\{f\in A\setminus\{0\}\vert \val(f)\succeq a\}\cup \{0\}$ (and similarly $F_{\val\succ a}$). 
The \emph{associated graded algebra} of the filtration $\{F_{\val\succeq a}\}_{a\in \mathbb Z^d}$ is 
\begin{eqnarray}\label{eq:def ass graded}
\gr_\val(A):=\bigoplus_{a\in \mathbb Z^d}F_{\val\succeq a}/F_{\val \succ a}.
\end{eqnarray}
If the filtered components $F_{\val \succeq a}/F_{\val \succ a}$ are at most one-dimensional for all $a\in\mathbb Z^d$, we say $\val$ has \emph{one-dimensional leaves}.
If $\val$ has full rank by \cite[Theorem~2.3]{KM16} it also has one-dimensional leaves.
Moreover, in this case there exists an isomorphism $\mathbb C[S(A,\val)]\cong \gr_\val(A)$  by \cite[Remark~4.13]{BG09}.
To define a $\mathbb Z_{\ge 0}$-filtration on $A$ induced by $\val$ we make use of the following standard trick that con be found in \cite[Proposition 1.8]{Ba82} or \cite[Lemma 3.2]{Cal02}.

\begin{lemma}\label{lem:caldero}
Let $F$ be a finite subset of $\mathbb Z^d$. Then there exists a linear form $e:\mathbb Z^d\to \mathbb Z_{\ge 0}$ such that for all $\bm,\bn\in F$ we have $\bm \prec \bn \Rightarrow e(\bm)>e(\bn)$ (note the switch!). 
\end{lemma}

Assume $S(A,\val)$ is finitely generated, more precisely assume it is generated by $\{\val(\bar x_{ij})\}_{i\in[s],j\in [k_i]}$.
In this case $\{\bar x_{ij}\}_{i\in[s],j\in [k_i]}$ is called a \emph{Khovanskii basis}\footnote{This term was introduced in \cite{KM16} generalizing the notion of SAGBI basis.} for $(A,\val)$.
Now choose a linear form as in the lemma for $F=\{\val(\bar x_{ij})\}_{i\in[s],j\in[k_i]}\subset \mathbb Z^d$.
We construct a $\mathbb Z_{\ge 0}$-filtration on $A$ by $F_{\le m}:=\{f\in A\setminus\{0\}\mid e(\val(f))\le m\}\cup\{0\}$ for $m\in \mathbb Z_{\ge 0}$.
Define similarly $F_{<m}$.
The associated graded algebra satisfies
\begin{eqnarray}
\gr_\val(A)\cong \bigoplus_{m\ge0}F_{\le m}/F_{<m}.
\end{eqnarray}
For $f\in A\setminus\{0\}$ denote by $\overline f$ its image in the quotient $F_{\le  e(\val(f))}/F_{<e(\val(f))}$, hence $\overline f\in\gr_\val(A)$. 
We obtain a family of $\mathbb C$-algebras containing $A$ and $\gr_\val(A)$ as fibers (see e.g. \cite[Proposition~5.1]{An13}) that can be defined as follows:

\begin{definition}\label{def: Rees algebra}
The \emph{Rees algebra} associated with the valuation $\val$ and the filtration $\{F_{\le m}\}_m$ is the flat $\mathbb C[t]$-subalgebra of $A[t]$ defined as
\begin{eqnarray}\label{eq: def Rees}
R_{\val,e}:=\bigoplus_{m\ge 0} (F_{\le m})t^m.
\end{eqnarray}
It has the properties that $R_{\val,e}/tR_{\val,e}\cong \gr_\val(A)$ and $R_{\val,e}/(1-t)R_{\val,e}\cong A$. 
In particular, if $A$ is $\mathbb Z_{\ge 0}$-graded, it defines a flat family over $\mathbb A^1$ (the coordinate on $\mathbb A^1$ given by $t$). 
The generic fiber is isomorphic to $\Proj(A)$ and the special fiber is the toric variety $\Proj(\gr_\val(A))$.
\end{definition}

It is desirable to have a valuation that encodes the grading of $A$: a valuation $\val$ on $A=\bigoplus_{m\in \mathbb Z^s_{\ge 0}} A_m$ is called \emph{homogeneous}, if for $f\in A$ with homogeneous decomposition $f=\sum_{m\in \mathbb Z^s_{\ge 0}} a_{m}f_m$ we have
\[
\val(f)=\min{}_\prec\{\val(f_m)\mid a_m\not=0\}.
\]
The valuation is called \emph{fully homogeneous}, if $\val(f)=(\deg f,\cdot)\in \mathbb Z^s\times \mathbb Z^{d-s}$ for all $f\in A$.
Given a full rank homogeneous valuations by \cite[Remark 2.6]{IW18} without loss of generality one may assume it is fully homogeneous.

Introduced by Lazarsfeld-Musta\c{t}\u{a} \cite{LM09} and Kaveh-Khovanskii \cite{KK12} we recall the definition of Newton--Okounkov body. 

\begin{definition}\label{def: val sg NO}
Let $\val:A\setminus \{0\}\to (\mathbb Z^s\times\mathbb Z^{d-s},\prec)$ be a fully homogeneous valuation. The \emph{Newton--Okounkov cone} is
\begin{eqnarray}\label{eq: def NO cone}
C(A,\val):=\overline{\cone(S(A,\val)\cup\{0\})}\subset \mathbb R^d. 
\end{eqnarray}
%
The \emph{Newton--Okounkov body} of $(A,\val)$ is now defined as
\begin{eqnarray}\label{eq: def NO body}
\Delta(A,\val):=C(A,\val) \cap \{(1,\dots,1)\}\times \mathbb R^{d-s}.
\end{eqnarray}
\end{definition}

For $\mathbb Z_{\ge 0}$-graded $A$, Anderson showed in \cite{An13} that if $\gr_{\val}(A)$ is finitely generated, $\Delta(A,\val)$ is a rational polytope. 
Moreover, it is the polytope associated to the normalization of the toric variety $\Proj(\gr_{\val}(A))$.

Dealing with polytopes throughout the paper we need the notion of \emph{Minkowski sum}. For two polytopes $A,B\subset \mathbb R^d$ it is defined as
\begin{eqnarray}\label{eq:Mink sum}
A+B:=\{a+b\mid  a\in A,b\in B\}\subset \mathbb R^d.
\end{eqnarray}
Consider polytpes of form $P_\val(\lambda):=C(A,\val)\cap (\{\lambda\}\times\mathbb R^{d-s})$ for $\lambda\in\mathbb R_{\ge 0}^s$.
It is a central question whether
$
P_\val(\epsilon_1)+\dots +P_\val(\epsilon_s)=\Delta(A,\val).
$
Our main result (Theorem~\ref{thm: val and quasi val with wt matrix} below) treats this question.

\section{Quasi-valuations with weighting matrices}\label{sec:val and quasival}

We briefly recall some background on higher-dimensional Gr\"obner theory  and quasi-valuations with weighting matrices as in \cite[\S3.1\&8.1]{KM16}. 
For classical results in Gr\"obner theory we rely on \cite[\S15]{Eis13}.
Then we define for a given valuation an associated quasi-valuation with weighting matrix.
The central result of this paper is Theorem~\ref{thm: val and quasi val with wt matrix}.
It is proved in full generality below, and applied to specific classes of valuations in \S\ref{sec:exp trop flag} and \S\ref{sec: exp plabic}.

As above let $A$ be a finitely generated algebra and domain of Krull-dimension $d$.
We assume $A$ has a presentation of form $\pi:S\to A$, such that $A=S/\ker(\pi)$. 
Here $S$ denotes the total coordinates ring of $\mathbb P^{k_1-1}\times\dots \times  \mathbb P^{k_s-1}$ with $\mathbb Z_{\ge 0}^s$-grading as defined above.
Let $I:=\ker(\pi)$ be homogeneous with respect to the $\mathbb Z_{\ge 0}^s$-grading.
To simplify notation, let $\pi(x_{ij})=:\bar x_{ij}$ for $i\in[s],j\in [k_i]$ and $n:=k_1+\dots +k_s$.
For simplicity we always consider (quasi-)valuations with image in $\mathbb Z^d$, but without much more effort all results extend to the case of $\mathbb Q^d$.

\begin{definition}\label{def: init wrt M}
Let $f=\sum a_{ u} x^{u} \in S$ with ${u}\in\mathbb Z^{n}$, where $x^{u}=x_1^{u_1}\cdots x_n^{u_n}$. For $M\in\mathbb Z^{d\times n}$ and a linear order $\prec$ on $\mathbb Z^d$ we define the \emph{initial form} of $f$ with respect to $M$ as
\begin{eqnarray}\label{eq: def init form M }
\init_M(f):=\sum_{\begin{smallmatrix}
Mm =\min_{\prec}\{M{u}\mid a_{u}\not=0\}
\end{smallmatrix}} a_{m} x^{m}.
\end{eqnarray}
We extend this definition to ideals $I\subset S$ by defining the \emph{initial ideal} of $I$ with respect to $M$ as $\init_M(I):=\langle \init_M(f) \mid f\in I \rangle\subset S$.
\end{definition}

Note that, by definition, if the ideal $I$ is homogeneous with respect to the $\mathbb Z^s_{\ge 0}$-grading, then so is every initial ideal of $I$.
Using the \emph{fundamental theorem of tropical geometry}  \cite[Theorem~3.2.3]{M-S} we recall the related notion of \emph{tropicalization} of an ideal $I\subset S$:
\begin{eqnarray}\label{eq:trop}
\trop(I)=
\left\{
     w\in\mathbb R^{n} 
        \mid     \init_{ w}(I) \text{  is monomial-free} 
\right\}.
\end{eqnarray}
By the \emph{Structure Theorem} \cite[Theorem 3.3.5]{M-S} we can choose a fan structure on $\trop(I)$ in such a way that it becomes a subfan of the Gr\"obner fan of $I$: 
$w$ and $v$ lie in the relative interior of a cone $C\subset \trop(I)$ (denoted $w,v\in C^\circ$) if $\init_w(I)=\init_v(I)$.
For this reason we adopt the notation $\init_C(I)$, which is defined as $\init_w(I)$ for arbitray $w\in C^\circ$.
Further, we call a cone $C\subset \trop(I)$ \emph{prime}, if $\init_C(I)$ is a prime ideal.

Coming back to the more general case of weighting matrices, we say that $M\in\mathbb Z^{d\times n}$ lies in the \emph{Gr\"obner region} GR$^d(I)$ of an ideal $I\subset S$, if there exists a monomial order\footnote{The initial form of an element $f=\sum a_ux^u\in S$ with respect to a monomial order $<$ is the monomial $a_mx^m$ of $f$ which satisfies $a_m\not =0$ and $x^m$ is minimal respect to $<$ among all monomials of $f$ with non-zero coefficient.} $<$ on $S$ such that 
\[
\init_{<}(\init_{M}(I))=\init_<(I).
\]
Such a monomial order is called \emph{compatible with $M$}.
If the ideal $I\subset S$ is (multi-)homogeneous with respect to the $\mathbb Z_{\ge 0}^s$-grading, then
by \cite[Lemma~8.7]{KM16} we have $\mathbb Q^{d\times n}\subset \text{GR}^d(I)$. 
To a given matrix $M\in\mathbb Z^{d\times n}$ one associates a quasi-valuation as follows. 
As above, fix a linear order $\prec$ on $\mathbb Z^d$.

\begin{definition}\label{def: quasi val from wt matrix}
Let $\tilde f=\sum a_{u}x^{u}\in S$ and define $\tilde\val_M:S\setminus\{0\}\to (\mathbb Z^d,\prec) $ by
$
\tilde \val_M(\tilde f):=\min{}_{\prec}\{M\bu\mid a_{\bu}\not=0\}.
$
By \cite[Lemma~3.2]{KM16}, there exists a quasi-valuation $\val_M:A\setminus\{0\}\to (\mathbb Z^d,\prec)$ given for $f\in A$ by
\[
\val_M(f):=\max{}_{\prec}\{\tilde\val_M(\tilde f)\mid \tilde f\in S, \pi(\tilde f)=f\}.
\]
It is called the \emph{quasi-valuation with weighting matrix} $M$. 
\end{definition}

We denote the associated graded algebra (defined analogously as in \eqref{eq:def ass graded} for valuations) of the quasi-valuation $\val_M$ by $\gr_M(A)$. It has the property
\begin{eqnarray}\label{eq: ass graded wt matrix val}
\gr_M(A)\cong  S/\init_M(I).
\end{eqnarray}
In particular, $\gr_M(A)$ inherits the $\mathbb Z^s_{\ge 0}$-grading of $S$, as $I$ is homogeneous with respect to this grading.



From the definition, it is usually hard to explicitly compute the values of a quasi-valuation $\val_M$.
The following proposition makes it more computable, given that $M$ lies in the Gr\"obner region of $I$.
Recall, that if $C_<\subset \text{GR}^d(I)$ is a maximal cone with associated monomial ideal $\init_<(I)$, then $\mathbb B:=\{\bar x^\alpha \mid x^\alpha\not \in \init_<(I)\}$ is a vector space basis for $A$, called \emph{standard monomial basis}.
The monomials $x^\alpha\not\in\init_<(I)$ are called \emph{standard monomials}.
In general, a vector space basis $\mathbb B\subset A$ is called \emph{adapted} to a valuation $\val:A\setminus\{0\}\to (\mathbb Z^d,\prec)$, if $F_{\val\succeq a}\cap\mathbb B$ is a vector space basis for $F_{\val\succeq a}$ for every $a\in\mathbb Z^d$.

\begin{proposition*}(\cite[Proposition~3.3]{KM16})
Let $M\in\text{GR}^d(I)$ and $\mathbb B\subset A$ be a standard monomial basis for the monomial order $<$ on $S$ compatible with $M$. 
Then $\mathbb B$ is adapted to $\val_M$. 
Moreover, for every element $f\in A$ written as $f=\sum \bar x^\alpha a_\alpha$ with $\bar x^\alpha\in \mathbb B$ and $a_\alpha\in \mathbb C$ we have
\begin{equation}\label{eq:compute valM}
\val_M(f)=\min{}_{\prec} \{ M\alpha\mid a_\alpha\not=0\}.
\end{equation}
\end{proposition*}

\begin{remark}\label{rmk:quasival homog}
The proposition implies that $\val_M$ is homogeneous: for $\alpha \in \mathbb Z_{\ge 0}^{k_1+\dots +k_s}$ define $m_\alpha:=(\sum_{j=1}^{k_1} \alpha_{1j},\dots,\sum_{j=1}^{k_s} \alpha_{sj})\in \mathbb Z_{\ge 0}^s$. 
Then $\bar x^\alpha \in A_{m_\alpha}$ for all $\bar x^\alpha\in \mathbb B$.
For $f\in A$ the unique expression $f=\sum \bar x^\alpha a_\alpha$ is also its homogeneous decomposition and by the proposition 
\[
\val_M(f)=\min{}_\prec\{\val_M(\bar x^\alpha)\mid a_\alpha\not =0\}.
\]
\end{remark}

From our point of view, quasi-valuations with weighting matrices are not the primary object of interest. In most cases we are given a valuation $\val:A\setminus\{0\}\to (\mathbb Z^d,\prec)$ whose properties we would like to know.
In particular, we are interested in the generators of the value semi-group and if there are only finitely many. 
The next definition establishes a connection between a given valuation and weighting matrices. It allows us to apply techniques from Kaveh-Manon for quasi-valuations with weighting matrices to other valuations of our interest.

\begin{definition}\label{def: wt matrix from valuation}
Given a valuation $\val:A\setminus\{0\}\to (\mathbb Z^d,\prec)$. 
We define the \emph{weighting matrix of $\val$} associated with the presentation $S/I$ of $A$ by
\[
M_\val:=(\val(\bar x_{ij}))_{i\in[s],j\in [k_i]} \in \mathbb Z^{d\times n}.
\]
That is, the columns of $M_\val$ are given by the images $\val(\bar x_{ij})$ for $i\in[s],j\in [k_i]$.
\end{definition}

Note we slightly abuse notation and write $M_\val$ instead of $M_\val(S/I)$ as we fix the presentation of $A$ from the start.
The following corollary is obtained by an argument very similar to the proof of \cite[Proposition~4.2]{KM16}. We therefore leave its proof to the reader. 

\begin{corollary}\label{cor: NO body val_M}
Let $M\in\mathbb Z^{d\times n}$ be of full rank with $d$ the Krull-dimension of $A$.
If $\init_{M}(I)$ is prime, then $\val_M$ is a valuation whose value semi-group $S(A,\val_M)$ is generated by $\{\val_M(\bar x_{ij})\}_{i\in[s],j\in[k_i]}$. 
In particular, the associated Newton--Okounkov body is given by
\[
\Delta(A,\val_M)=\sum_{i=1}^s\conv(\val_M(\bar x_{ij})\mid j\in [k_i]).
\]
\end{corollary}



To a full rank valuation $\val:A\setminus\{0\}\to (\mathbb Z^d,\prec)$ we associate a homomorphism of polynomial rings, called \emph{monomial map of $\val$} (see e.g. \cite{MoSh}):
\[
\phi_\val: S \to \mathbb C[y_1,\dots,y_d] \quad \text{ by } \quad \phi_\val(x_{ij}):= y^{\val(\bar x_{ij})}.
\] 
The image $\text{im}(\phi_\val)$ is naturally isomorphic to the semi-group algebra $\mathbb C[S(\val(\bar x_{ij})_{ij})]$, where $S(\val(\bar x_{ij})_{ij})$ is the semi-group generated by $\{\val(\bar x_{ij})\vert i\in[s],j\in [k_i]\}\subset \mathbb Z^d$.
We have
\begin{eqnarray}\label{eq:kernel phi and gr}
S/\ker(\phi_\val)\cong \mathbb C[S(\val(\bar x_{ij})_{ij})]\subset \mathbb C[S(A,\val)]\cong \gr_\val(A).
\end{eqnarray}
Therefore, $\ker(\phi_\val)$ is a binomial prime ideal. The height of $\ker(\phi_\val)$ is $n-\text{rank}(S(\val(\bar x_{ij})_{ij}))=n-\text{rank}(M_\val)$. 

\begin{lemma}\label{lem: wt matrix val in trop}
For every  full rank valuation  $\val:A\setminus\{0\}\to (\mathbb Z^d,\prec)$ we have $\init_{M_\val}(I)\subset \ker(\phi_\val)$.
\end{lemma}

\begin{proof}
Let $f\in I$ and consider $\init_{M_\val}(f)=\sum_{i=1}^s c_ix^{u_i}$ with $c_i\in \mathbb C$ and $u_i\in \mathbb Z_{\ge 0}^n$ for all $i$.
So $Mu_i=Mu_j=:a$, i.e. $\val(\bar x^{u_i})=\val(\bar x^{u_j})$ and $\bar x^{u_i},\bar x^{u_j}\in F_{\val \succeq a}$ for all $1\le i,j\le s$.
As $\val$ has one-dimensional leaves, we have $\sum_{i=1}^s c_i\bar x^{u_i}\in F_{\val \succ a}$, which
implies it is zero in $\gr_\val(A)$.
Hence, by \eqref{eq:kernel phi and gr} we have $\init_{M_\val}(f)\in \ker(\phi_\val)$. 
\end{proof}


Before stating the main theorem relating a given valuation $\val$ with the (quasi-)valuation with weighting matrix $M_\val$ we prove the following proposition that is used in the proof. 

\begin{proposition}\label{prop:same val}
Assume $I$ is homogeneous with respect to the $\mathbb Z_{\ge 0}^s$-grading and Krull dimension of $A=S/I$ is $d$. 
Let $\val:A\setminus\{0\}\to (\mathbb Z^d,\prec)$ be a full rank valuation
with $M_\val\in\mathbb Z^{d\times n}$ the weighting matrix of $\val$. 
Then
\[
\val=\val_{M_\val} \quad \quad \Leftrightarrow \quad \quad \init_{M_\val}(I) \text{ is prime} \text{ and }  \text{rank}(M_\val)=d.
\]
\end{proposition}
\begin{proof}
\begin{itemize}
    \item[``$\Rightarrow$"] Assume $\val=\val_M$, then $\mathbb C[S(A,\val)]=S/\init_{M_{\val}}(I)$ by \eqref{eq: ass graded wt matrix val}. In particular, $\init_{M_\val}(I)$ is binomial and prime. Further, $\text{rank}(M_\val)=\text{rank}(S(A,\val))=\text{Krull dimension of }A$.
    
    \item[``$\Leftarrow$"] Assume $\init_{M_\val}(I)$ is prime and $M_\val$ has rank equal to the Krull dimension of $A$.
    By \cite[Theorem 2.17]{KM16} we obtain $\val=\val_{M_{\val}}$ if and only if $\init_{M_\val}(I)=\ker \phi_\val$. By Lemma~\ref{lem: wt matrix val in trop} we have $\init_{M_\val}(I)\subset \ker \phi_\val$. As $\text{rank}(M_\val)=\text{rank}(S(A,\val))$ both ideals have the same height and are therefore equal.
\end{itemize}
\vspace{-.5cm}
\end{proof}

\begin{theorem}\label{thm: val and quasi val with wt matrix}
Assume $I$ is homogeneous with respect to the $\mathbb Z_{\ge 0}^s$-grading and Krull dimension of $A=S/I$ is $d$. 
Let $\val:A\setminus\{0\}\to (\mathbb Z^d,\prec)$ be a full rank valuation
with associated weighting matrix $M_\val\in\mathbb Z^{d\times n}$.
Then, 
\[
S(A,\val)\quad \text{is generated by} \quad \{\val(\bar x_{ij})\}_{i\in[s],j\in[k_i]} \quad \quad \Leftrightarrow \quad \quad \init_{M_\val}(I) \quad \text{is prime and }\text{rank}(M_\val)=d. 
\]
\end{theorem}

Both directions of the proof rely on the equality $\val=\val_{M_\val}$. 
The equality is false in general, a counterexample is given in \S\ref{sec:exp trop flag} Example~\ref{exp:val vs valM}.
For ``$\Leftarrow$" the equality follows from Proposition~\ref{prop:same val} while for ``$\Rightarrow$" it follows from a direct computation.

\begin{proof}
As $I$ is homogeneous with respect to a positive grading, $M_\val$ lies in the Gr\"obner region of $I$. 
Let $\mathbb B\subset A$ be the standard monomial basis adapted to $\val_{M_\val}$.

\begin{itemize}
\item[``$\Leftarrow$"] Assume $\init_{M_\val}(I)$ is prime and $M_\val$ has full rank. Then we obtain by Proposition~\ref{prop:same val} $S(A,\val)=S(A,\val_{M_\val})\cong S/\init_{M_\val}(I)$.
Hence, by Corollary~\ref{cor: NO body val_M} $S(A,\val)$ is generated by $\{\val(\bar x_{ij})\}_{ij}$.

\item[``$\Rightarrow$"] 
Assume $S(A,\val)$ is generated by $\{\val(\bar x_{ij})\}_{i\in[s],j\in[k_i]}$.
We have 
$\val_{M_\val}(\bar x^\alpha)=M_\val \alpha=\val(\bar x^\alpha)$ for $\bar x^\alpha \in \mathbb B$ by \eqref{eq:compute valM}.
We need to prove $\val=\val_{M_{\val}}$ then $\init_{M_\val}(I)$ is prime and $\text{rank}(M_\val)=$ Krull dimension of $A$ by Proposition~\ref{prop:same val}. 
This follows from:

\smallskip
\emph{Claim:} $\val_{M_\val}(\bar x^{\alpha})\not =\val_{M_\val}(\bar x^{\beta})$ for $\bar x^{\alpha},\bar x^{\beta}\in \mathbb B$ with $\alpha\not=\beta$.

\smallskip
\noindent
\emph{Proof of claim:} Assume there exist $\bar x^{\alpha},\bar x^{\beta}\in \mathbb B$ with $\alpha\not=\beta$ and $\val_{M_\val}(\bar x^{\alpha})=\val_{M_\val}(\bar x^{\beta})$.
Then
\[
\val(\bar x^{\alpha})=M_\val\alpha =\val_{M_\val}(\bar x^{\alpha})=\val_{M_\val}(\bar x^{\beta})=M_\val\beta=\val(\bar x^{\beta}).
\]
This implies that $\val(\bar x^{\alpha}+\bar x^{\beta})\succ \val(\bar x^{\alpha})=\val(\bar x^{\beta})$.
In particular, $\val(\bar x^{\alpha}+\bar x^{\beta})$ does not lie in the semi-group span $\langle \val(\bar x_{ij})\mid i\in [s],j\in [k_i]\rangle=S(A,\val)$, a contradiction by assumption.
\end{itemize}
\vspace{-.65cm}
\end{proof}

Knowing the generators of $S(A,\val)$ \emph{explicitly} has a number of crucial consequences for applications. Given the theorem, they now follow directly from $\init_{M_\val}(I)$ being prime:

\begin{corollary}
Assume $I$ is homogeneous with respect to the $\mathbb Z_{\ge 0}^s$-grading and generated by elements $f\in I$ with $\deg f>\epsilon_i$ for all $i\in[s]$.
Let $\val:A\setminus\{0\}\to (\mathbb Z^d,\prec)$ be a full rank valuation
with $M_\val\in\mathbb Z^{d\times n}$ of full rank. Assume additionally that $\init_{M_\val}(I)$ is prime. 
Then $\val$ is homogeneous and
\begin{itemize}
    \item[(i)] $\gr_\val(A)\cong S/\init_{M_\val}(I)$, 
    \item[(ii)] $\{\bar x_{ij}\}_{i\in[s],j\in[k_i]}$ is a Khovanskii basis for $(A,\val)$, and
    \item[(iii)] $\Delta(A,\val)=\sum_{i=1}^s\conv(\val(\bar x_{ij})\mid j\in[k_i])$.
\end{itemize}
\end{corollary}


\subsection{From weighting matrix to weight vector and tropicalization}\label{sec:wt matrix wt vector}
In this section we summarize how to pass from a weighting matrix to a weight vector which is desirable for applications.
We assume the ideal $I\subset S$ is homogeneous with respect to the $\mathbb Z^d$-grading on $S$ and that $A=S/I$ has Krull dimension $d$.
We consider weighting matrices that lie in the Gr\"obner region of $I$ and for simplicity we assume they have integer entries.

\begin{definition}
Let $M\in \mathbb Z^{d\times n}$ and denote by $M_1,\dots,M_n$ its columns.
A linear map $e:\mathbb Z^d \to \mathbb Z$ is called an \emph{order preserving projection} with respect to $M$ and $I$, if for $eM:=(e(M_1),\dots,e(M_n))$ we have $\init_M(I)=\init_{eM}(I)$.
\end{definition}

The following lemma a reformulation of classical results in Gr\"obner theory (see e.g. \cite{Ba82,Cal02,KM16}). We include it for completeness in view of applications.

\begin{lemma}\label{lem: wt for matrix}
For every $M\in\mathbb Z^{d\times n}$ the order preserving projections with respect to $M$ and $I$ are organized in a polyhedral cone $\sigma^\circ_M(I)\subset \mathbb R^d$.
Moreover, there exist $w\in \mathbb R^n$ with $\init_w(I)=\init_M(I)$.
\end{lemma}

\begin{proof}
Let $\{R_1,\dots,R_s\}$ be a reduced Gr\"obner basis for $\init_M(I)$. Assume the initial form of $R_l$ is of weight $Mm_l$ for $l\in [s]$.
So we have $R_l=\init_M(R_l) + \sum_{j=1}^{q^l}a^l_jx^{u^l_j}$ with $Mm_l\prec Mu_j^l$ for all $j\in [q^l]$.
We define the (open) polyhedral cone of order preserving projections
\begin{eqnarray}\label{eq:cone of order pres proj}
\sigma_M^\circ:=\{e\in \mathbb R^d\mid \langle e, M(m_l-u_j^l) \rangle <0 \text{ for all } l\in[s], j\in [q^l]\}. 
\end{eqnarray}
Applying the linear map defined by $M$ we obtain a set $E_M^\circ:=\{eM\mid e\in \sigma_M^\circ\}$ and
\begin{eqnarray}\label{eq:def cone for M}
E_M \subset \{w\in \mathbb R^n\mid \langle w,m_l-u_j^l\rangle <0 \text{ for all } l\in [s],j\in [q^l]\}=:C_M^\circ. 
\end{eqnarray}
By Lemma~\ref{lem:caldero}, $\sigma_M^\circ$ is nonempty, hence $C_M^\circ$ is nonempty and by definition an open cone in the Gr\"obner fan of $I$.  
\end{proof}

\begin{corollary}\label{cor:wt matrix in trop}
Consider $\val:A\setminus \{0\}\to (\mathbb Z^d,\prec)$ a full rank valuation with associated weighting matrix $M_\val\in \mathbb Z^{d\times n}$.
Then $\init_{M_\val}(I)$ is monomial-free and $C_{M_\val}$ is a cone in $\trop(I)$.
\end{corollary}
\begin{proof}
By Lemma~\ref{lem: wt matrix val in trop} we have $\init_{M_\val}(I)$ is contained in a binomial prime ideal. 
\end{proof}

\section{Application: (tropical) flag varieties and string cones}\label{sec:exp trop flag}

In this subsection we focus on a particular valuation on the \text{Cox} ring of the full flag variety. 
The valuation was defined in the context of representation theory by \cite{FFL15} (see also \cite{Kav15}).
It is of particular interest, as it realizes Littelmann's string polytopes \cite{Lit98,BZ01} as Newton--Okounkov polytopes.
Applying Theorem~\ref{thm: val and quasi val with wt matrix} to this valuation, we solve a conjecture by \cite{BLMM} relating the Minkowski property of string cones (see Definition~\ref{def:MP}) to the tropical flag variety.

\medskip
Consider $\Flag_n$, the full flag variety of flags of vector subspaces of $\mathbb C^n$. 
Its dimension as a projective variety is $N:=\frac{n(n-1)}{2}$.
We embed it into the product of Grassmannians $\Gr_1(\mathbb C^n)\times\dots\times \Gr_{n-1}(\mathbb C^n)$.
Further embedding each Grassmannian via its Pl\"ucker embedding into projective space we obtain
$\Flag_n\hookrightarrow \mathbb P^{\binom{n}{1}-1}\times\dots \times \mathbb P^{\binom{n}{n-1}-1}$.
Set $K:=\binom{n}{1}+\dots + \binom{n}{n-1}$.
In this way, we get the (multi-)homogeneous coordinate ring $A_n$ as a quotient of the total coordinate ring of the product of projective spaces $S:=\mathbb C[p_J\mid 0\not =J\subsetneq [n]]$.
Namely $A_n\cong S/I_n$.
In this case $S$ is multigraded by $\mathbb Z_{\ge 0}^{n-1}$: $\deg(p_J):=\epsilon_{k}$ for $\vert J\vert=k$ and $\{\epsilon_k\}_{k=1}^{n-1}$ standard basis of $\mathbb Z^{n-1}$.
The variables $p_J$ are called \emph{Pl\"ucker variables} and their cosets $\bar p_J\in A_n$ are \emph{Pl\"ucker coordinates}.
The ideal $I_n\subset S$ is generated by the quadratic \emph{Pl\"ucker relations}, see \cite[Theorem~14.6]{MS05}.
It is homogeneous with respect to the $\mathbb Z_{\ge 0}^{n-1}$-grading induced from $S$.
The \emph{tropical flag variety}, denoted $\trop(\Flag_n)$ (as defined in \cite{BLMM}) is the tropicalization of the ideal $I_n\subset S$.
\medskip

Realizing $\Flag_n=SL_n/B$, where $B\subset SL_n$ are upper triangular matrices, one can make use of the representation theory of $SL_n$ to define valuations on $A_n$.
We summarize the necessary representation-theoretic background below.

Let $S_n$ denote the symmetric group and $w_0\in S_n$ its longest element.
By $\w_0$ we denote a reduced expression $s_{i_1}\ldots s_{i_N}$ of $w_0$ in terms of simple transpositions $s_i:=(i,i+1)$. 
Let $\lie{sl}_n$ be the Lie algebra of $SL_n$ and fix a Cartan decomposition $\lie{sl}_n=\lie{n^-}\oplus\lie{h}\oplus\lie{n}^+$ into lower, diagonal and upper triangular traceless matrices.
Let $R=\{\varepsilon_i-\varepsilon_j\}_{i\neq j}$ be the root system of type $\mathtt A_{n-1}$, where $\{\varepsilon_i\}_{i=1}^n$ is the standard basis of $\mathbb R^n$.
Every $\varepsilon_i-\varepsilon_{i+1}$ defines an element\footnote{The element $f_{i}\in \lie{n}^-$ is the elementary matrix with only non-zero entry $1$ in the $(i+1,i)$-position.} $f_{i}\in \lie{n}^-$.

Denote the weight lattice by $\Lambda\cong \mathbb Z^{n-1}$. 
It is spanned by the fundamental weights $\omega_1,\dots,\omega_{n-1}$ and we set $\Lambda^+:=\mathbb Z_{\ge 0}^{n-1}$ with respect to the basis of fundamental weights.
For every $\lambda\in \Lambda^+$ we have an irreducible highest weight representation $V(\lambda)$ of $\lie{sl}_n$.
It is cyclically generated by a highest weight vector $v_\lambda$ (unique up to scaling) over the universal enveloping algebra $U(\lie{n}^-)$.
In particular, for every reduced expression $\w_0=s_{i_1}\cdots s_{i_{N}}$ the set $\mathcal S_{\w_0}:=\{ f_{i_1}^{m_{i_1}} \cdots f_{i_{N}}^{m_{i_{N}}}(v_\lambda) \in V(\lambda) \mid m_{i_j} \geq 0 \}$
is a spanning set for the vector space $V(\lambda)$.

\begin{example}\label{exp:fund rep}
Given a fundamental weight $\omega_k$, we have $V(\omega_k)=\bigwedge^k\mathbb C^n$. 
We can choose the highest weight vector as $v_{\omega_k}:=e_1\wedge\dots\wedge e_k$, where $\{e_i\}_{i=1}^n$ is a basis of $\mathbb C^n$.
In this context Pl\"ucker coordinates $\bar p_{\{j_1,\dots,j_k\}}\in A_n$ are identified with dual basis elements $(e_{j_1}\wedge\dots\wedge e_{j_k})^*\in V(\omega_k)^*$.
\end{example}

Littelmann \cite{Lit98} introduced in the context of quantum groups and crystal bases the so called (weighted) string cones and string polytopes $Q_{\w_0}(\lambda)$.
The motivation is to find monomial bases for the irreducible representations $V(\lambda)$.
Littelmann identifies a linearly independent subset of the spanning set $\mathcal S_{\w_0}$ by introducing the notion of \emph{adapted string} (see \cite[p.~4]{Lit98}) referring to a tuple $(a_1,\dots,a_{N})\in\mathbb Z_{\ge 0}^{N}$.
His basis for $V(\lambda)$ consists of those elements $f_{i_1}^{a_1}\cdots f_{i_{N}}^{a_{N}}(v_\lambda)$ for which $(a_1,\dots,a_{N})$ is adapted.

For a fixed reduced expression $\w_0$ of $w_0\in S_n$ and $\lambda\in\Lambda^+$ he gives a recursive definition of the the \emph{string polytope} $Q_{\w_0}(\lambda)\subset\mathbb R^{N}$ (\cite[p.~5]{Lit98}, see also \cite{BZ01}). 
The lattice points $Q_{\w_0}(\lambda)\cap \mathbb Z^{N}$ are the adapted strings for $\w_0$ and $\lambda$.
The \emph{string cone} $Q_{\w_0}\subset \mathbb R^{N}$ is the convex hull of all $Q_{\w_0}(\lambda)$ for $\lambda\in\Lambda^+$. The \emph{weighted string cone} is defined as
\[
\mathbf Q_{\w_0} := \conv \big( \bigcup{}_{\lambda\in\Lambda^+} \{\lambda\} \times Q_{\w_0}(\lambda)\big)\subset \mathbb R^{n-1+N}.
\]
By definition, one obtains the string polytope from the weighted string cone by intersecting it with the hyperplanes $ \{\lambda\} \times \mathbb R^N$.
In this context, a central question is the following: 
\begin{center}
\emph{Given $\w_0$, does $Q_{\w_0}(\lambda)=a_1Q_{\w_0}(\omega_1)+\dots+a_{n-1}Q_{\w_0}(\omega_{n-1})$ hold for all $\lambda=\sum_{k=1}^{n-1}a_{k}\omega_k\in \Lambda^+$?}
\end{center}

\noindent
\begin{definition}\label{def:MP}
A reduced expression $\w_0$ of $w_0\in S_n$ has the \emph{(strong) Minkowski property}, if $Q_{\w_0}(\lambda)=a_1Q_{\w_0}(\omega_1)+\dots+a_{n-1}Q_{\w_0}(\omega_{n-1})$ holds for all $\lambda=\sum_{k=1}^{n-1}a_{k}\omega_k\in \Lambda^+$.

If $Q_{\w_0}(\rho)=Q_{\w_0}(\omega_1)+\dots+Q_{\w_0}(\omega_{n-1})$ for $\rho:=\omega_1+\dots+\omega_{n-1}$, we say ${\w_0}$ \emph{satisfies MP}. 
\end{definition}

String polytopes can be realized as Newton--Okounkov polytopes, as was done by \cite{Kav15,FFL15}.
The corresponding valuations are defined on the coordinate ring of $G //
U:=\Spec(\mathbb C[G]^U)$, where $U\subset B$ are upper triangular matrices with $1$'s on the diagonal.
By \cite{PV89} $\mathbb C[G// U]$ is isomorphic to the Cox ring $\text{Cox}(\Flag_n)$ of the flag variety.
As a basis for the Picard group we fix the homogeneous line bundle associated to the fundamental weights $L_{\omega_1},\dots,L_{\omega_{n-1}}$ (see e.g. \cite[p.15]{Bri04}).
This way we obtain 
\[
\text{Cox}(\Flag_n)=\bigoplus_{\lambda\in\Lambda^+} H^0(\Flag_n,L_\lambda)\cong \bigoplus_{\lambda \in \Lambda^+} V(\lambda)^*,
\]
where the isomorphism is due to the Borel-Weil theorem.
For every $\w_0$ there exists a fully homogeneous valuation $\val_{\w_0}:\text{Cox}(\Flag_n) \setminus \{0\}\to \Lambda^+\times\mathbb Z^{N}$ called \emph{string valuation}.
By \cite[\S11]{FFL15} it has the properties
\begin{eqnarray}\label{eq:string and val}
C(\text{Cox}(\Flag_n), \val_{\w_0})=\mathbf Q_{\w_0} \quad  \text{ and } \quad P_{\val_{\w_0}}(\lambda)=Q_{\w_0}(\lambda) \text{ for all } \lambda\in \Lambda^+.
\end{eqnarray}
%

As described above, we fixed $\Flag_n\hookrightarrow \mathbb P^{\binom{n}{1}-1}\times \dots \times \mathbb P^{\binom{n}{n-1}-1}$ and want to consider the string valuations on the homogeneous coordinate ring $A_n$.
The following proposition is therefore very useful. 
As a consequence of the standard monomial theory for flag varieties it follows from \cite[Theorem 12.8.3]{LB09}. 
We give the proof for completion.
To simplify notation let $Z:=\mathbb P^{\binom{n}{1}-1}\times \dots \times \mathbb P^{\binom{n}{n-1}-1}$.
Recall that Schubert varieties are of form $X(w):=\overline{BwB}/B\subset SL_n/B$ for $w\in S_n$ and that for $w=w_0s_k$ they are divisors in $\Flag_n$.

\begin{proposition}\label{prop:Cox}
We have $A_n \cong \text{Cox}(\Flag_n)$.
\end{proposition}

\begin{proof}
Let $L_k$ be the ample generator of $\text{Pic}(\mathbb P^{\binom{n}{k}-1})$. 
By \cite[Theorem 11.2.1]{LB09} it pulls back to the line bundle $L_{\omega_k}$ of the Schubert divisor $X(w_0s_{k})\subset\Flag_n\hookrightarrow Z$ and $\text{Pic}(\Flag_n)$ is generated by $\{L_{\omega_k}\}_{k=1}^{n-1}$. 
For every $\lambda=\sum_{k=1}^{n-1}a_k\omega_k\in \mathbb Z^{n-1}_{\ge 0}$ by \cite[Theorem 12.8.3]{LB09} we have a surjection
\[
H^0(Z,L_1^{a_1}\otimes\dots \otimes L_{n-1}^{a_{n-1}})\to H^0(\Flag_n,L_\lambda).
\]
Note that $S=\bigoplus_{a\in \mathbb Z^{n-1}_{\ge 0}} H^0(Z,L_1^{a_1}\otimes\dots \otimes L_{n-1}^{a_{n-1}})$, hence we have a surjection $\pi: S\to \text{Cox}(\Flag_n)$.
Its kernel is the multi-homogeneous ideal generated by elements $f\in H^0(Z,L_1^{a_1}\otimes\dots \otimes L_{n-1}^{a_{n-1}})$ vanishing on $\Flag_n$ for varying $(a_1,\dots,a_{n-1})\in \mathbb Z^{n-1}_{\ge 0}$.
In particular, $\ker (\pi)=I_n$ and $\text{Cox}(\Flag_n)\cong S/I_n=A_n$.
\end{proof}

From now on we consider for a fixed reduced expression $\underline w_0=s_{i_1}\ldots s_{i_N}$ the valuation $\val_{\w_0}$ on $A_n$. 
We explain how to compute $\val_{\w_0}$ on Pl\"ucker coordinates, following \cite{FFL15}.
We use the total order $\prec$ on $\mathbb Z^N$ defined by
$m\prec n$ if and only if, $\sum_{i=1}^N m_i< \sum_{i=1}^N n_i$ or $\sum_{i=1}^N m_i = \sum_{i=1}^N n_i$  and  ${ m}>_{\text{lex}}{n}$.
Now $\val_{\w_0}$ can be computed explicitly\footnote{The action of $\lie{n^-}$ on $\bigwedge^k\mathbb C^n$ necessary to make explicit computations is given by $f_{i,j}(e_{i_1}\wedge\dots\wedge e_{i_k})= \sum_{s=1}^k e_{i_1}\wedge\dots\wedge (f_{i,j}e_{i_s})\wedge \dots \wedge e_{i_k}$, where $f_{i,j}e_{i_s}=\delta_{i,i_s}e_{j+1}$ and $f_i=f_{i,i+1}$.} on Pl\"ucker coordinates by \cite[Proposition~2]{FFL15}: 
\[
\val_{\w_0}(\bar p_{j_1,\dots,j_k})=\min{}_{\prec}\{\bm \in\mathbb Z^{N}_{\ge 0}\mid {\bf f^m}(e_{1}\wedge\dots\wedge e_{k})=e_{j_1}\wedge\dots\wedge e_{j_k}\},
\]
for $\{j_1,\dots,j_k\}\subset [n]$, where ${\bf f^m}=f_{{i_1}}^{m_1}\cdots f_{{i_N}}^{m_N}\in U(\mathfrak{n}^-)$.

We want to apply the results from \S\ref{sec:val and quasival} to the given valuations of form $\val_{\w_0}:A_n\setminus\{0\}\to (\mathbb Z^N,\prec)$.
Let therefore $M_{\w_0}:=(\val_{\w_0}(\bar p_J))_{0\not=J\subsetneq [n]}\in\mathbb Z^{N\times K}$ be the matrix whose columns are given by the images of Pl\"ucker coordinates under $\val_{\w_0}$.
Consider as in \cite{BLMM} the linear form $e:\mathbb Z^N\to\mathbb Z$ given by
\[
-e({\bf m}):=2^{N-1}m_1+2^{N-2}m_2+\ldots +2m_{N-1}+m_N.
\]

\begin{definition}\label{str.wt.vec}
For a fixed reduced expression ${\underline w_0}$ the \emph{weight} of the Pl\"ucker variable $p_J$ is $e(\val_{\w_0}(\bar p_J))$. We define the \emph{weight vector} ${\bf w}_{\underline w_0}:=(e(\val_{\w_0}(\bar p_J)))_{0\not =J\subsetneq [n]}\in \mathbb R^K$.
\end{definition}

The weight vectors ${\bf w}_{\w_0}$ are computed in \cite{BLMM} for $n=4$ and $n=5$. The authors verify that they lie in the tropical flag variety and conjecture this is true in general. 

\begin{lemma}\label{lem:wt in trop flag}
For every reduced expression $\w_0$ we have $\init_{M_{\w_0}}(I_n)=\init_{{\bf w}_{\w_0}}(I_n)$. In particular, ${\bf w}_{\w_0}\in \trop(\Flag_n)$.
\end{lemma}

\begin{proof}
The linear map $e:\mathbb Z^N\to \mathbb Z$ is constructed using the recursive recipe given in \cite[Proof of Lemma 3.2]{Cal02}.
Caldero uses it to define the filtration $\{\mathcal F_m\}_{m\in \mathbb Z}$ on $A_n$.
In Proposition~3.1 and Corollary~3.2 of \cite{Cal02} he shows that this $\mathbb Z$-filtration and the $\mathbb Z^{N+(n-1)}$-filtration $\{F_{\val_{\w_0}}\}$ have the same associated graded algebra, namely $\mathbb C[S(A_n,\val_{\w_0})]$.
To be precise, given a relation in $A_n$ coming from a lifting of a minimal relation in $S(A_n,\val_{\w_0})$ Caldero shows that in the associated graded algebras with respect to $\{\mathcal F_m\}_{m\in \mathbb Z}$ and  $\{F_{\val_{\w_0}}\}$ the same initial relation holds.
The generators of $I_n$ are naturally liftings of minimal relations in $S(A_n,\val_{\w_0})$ as they are of minimal total degree $2$. 
So in particular, $\init_{M_{\w_0}}(I_n)=\init_{{\bf w}_{\w_0}}(I_n)$. The rest follows from Corollary~\ref{cor:wt matrix in trop}.
\end{proof}

\begin{center}
\begin{table}
\centering
\begin{tabular}{| c | c | c | c |}
\hline
$\underline w_0$ & MP & weight vector $-{\bf w}_{\underline w_0}$& $\init_{{\bf w}_{\w_0}}(I_4)$ prime \\
\hline
\begin{tabular}{ l}
String 1: \\
$s_1s_2s_1s_3s_2s_1$ \\
$s_2s_1s_2s_3s_2s_1$ \\
$s_2s_3s_2s_1s_2s_3$ \\
$s_3s_2s_3s_1s_2s_3$
\end{tabular} &
\begin{tabular}{ l}
 \\
yes \\
yes \\
yes \\
yes
\end{tabular}
&
\begin{tabular}{ l}
 \\
$(0, 32, 24, 7, 0, 16, 6, 48, 38, 30, 0, 4, 20, 52)$ \\
$(0, 16, 48, 7, 0, 32, 6, 24, 22, 54, 0, 4, 36, 28)$ \\
$(0, 4, 36, 28, 0, 32, 24, 6, 22, 54, 0, 16, 48, 7)$ \\
$(0, 4, 20, 52, 0, 16, 48, 6, 38, 30, 0, 32, 24, 7)$
\end{tabular}
&
\begin{tabular}{ l}
 \\
yes \\
yes \\
yes \\
yes
\end{tabular} \\
\hline
\begin{tabular}{ l}
String 2: \\
$s_1s_2s_3s_2s_1s_2$ \\
$s_3s_2s_1s_2s_3s_2$
\end{tabular} &
\begin{tabular}{ l}
  \\
yes \\
yes
\end{tabular}&
\begin{tabular}{ l}
  \\
$(0, 32, 18, 14, 0, 16, 12, 48, 44, 27, 0, 8, 24, 56)$ \\
$(0, 8, 24, 56, 0, 16, 48, 12, 44, 27, 0, 32, 18, 14)$
\end{tabular}
&
\begin{tabular}{ l}
  \\
yes \\
yes
\end{tabular} \\
\hline
\begin{tabular}{ l}
String 3: \\
$s_2s_1s_3s_2s_3s_1$
\end{tabular}
&
\begin{tabular}{ l}
  \\
yes
\end{tabular}
&
\begin{tabular}{ l}
  \\
$(0, 16, 48, 13, 0, 32, 12, 20, 28, 60, 0, 8, 40, 22)$
\end{tabular}&
\begin{tabular}{ l}
  \\
yes
\end{tabular}\\
\hline

\begin{tabular}{ l}
String 4: \\
$s_1s_3s_2s_1s_3s_2$ 
\end{tabular}&
\begin{tabular}{ l}
  \\
no
\end{tabular}
&
\begin{tabular}{ l}
  \\
$(0, 16, 12, 44, 0, 8, 40, 24, 56, 15, 0, 32, 10, 26)$
\end{tabular}&
\begin{tabular}{ l}
  \\
no
\end{tabular}\\
\hline
\end{tabular}
\caption{Isomorphism classes of string polytopes (up to unimodular equivalence) for $n=4$ and $\rho$ depending on $\underline w_0$, the property MP, the weight vectors ${\bf w}_{\underline w_0}$ as in Definition~\ref{str.wt.vec}, and primeness of the initial ideals $\init_{{\bf w}_{\underline w_0}}(I_4)$. For details see \cite[\S5]{BLMM}}  
\label{tab:string4}
\end{table}
\end{center}
\vspace{-1cm}
\begin{example}
Consider the reduced expression $\hat\w_0=s_1s_2s_1s_3s_2s_1$ for $w_0\in S_4$.
We compute $\val_{\hat\w_0}(\bar p_{13})=(0,1,0,0,0,0)$ and the weight of $\bar p_{13}$: $e(0,1,0,0,0,0)=-(1\cdot 2^4)=-16$. 
Similarly, we obtain weights for all Pl\"ucker variables and (ordered lexicographically by their indexing sets)
\[
-{\bf w}_{\underline w_0}=(0,32,24,7,0,16,6,48,38,30,0,4,20,52).
\]
Table~\ref{tab:string4} contains all weight vectors (up to sign) for $\Flag_4$ constructed this way.\end{example}

\begin{lemma}\label{lem:full rank}
For every reduced decomposition $\w_0$ of $w_0\in S_n$ we have ${\rm rank}(M_{\w_0})=N+n-1$, i.e $M_{\w_0}$ is of full rank.
\end{lemma}
\begin{proof}
We first restrict our attention to a submatrix of $M_{\w_0}$ having only those columns corresponding to $\val_{\w_0}(\bar p_I)$ for $I\in \binom{[n]}{k}$, denote it by $M_{\w_0}\vert_{|I|=k}$ for any $k\in[n-1]$. 
This matrix corresponds to a full-rank valuation on the homogeneous coordinate ring of the Grassmannian under the Pl\"ucker embedding.
So, ${\rm rank}(M_{\w_0}\vert_{|I|=k})=k(n-k)+1$.
In particular, we deduce for $n\ge 4$ 
\[
{\rm rank}(M_{\w_0})=\min \left\{\sum_{k=1}^{n-1}(k(n-k)+1),N+n-1\right\}=N+n-1.
\]
Direct computations reveals the statement is also true for $n\le 3$.
\end{proof}

Based on computational evidence for $n\le 5$ the authors of \cite{BLMM} state the following conjecture relating the Minkowski property with the tropical flag variety.

\begin{conjecture}
For $n\ge 3$ and $\w_0$ a reduced expression of $w_0\in S_n$ we have ${\bf w}_{\w_0}\in \trop(\Flag_n)$.
Moreover, if ${\bf w}_{\w_0}$ lies inside the relative interior of a maximal prime cone of $\trop(\Flag_n)$, then ${\w_0}$ satisfies MP.
\end{conjecture}

We already proved the first part of the conjecture in Lemma~\ref{lem:wt in trop flag}.
The second part is a consequence of Theorem~\ref{thm: val and quasi val with wt matrix} as we explain below.

\begin{theorem}\label{thm: quasival for string}
Let $\w_0$ be a reduced expression of $w_0\in S_n$ and consider $\bw_{\w_0}\in\mathbb R^{K}$. 
If $\init_{\bw_{\w_0}}(I_n)$ is prime, then $S(A_n,\val_{\w_0})$ is generated by $\{\val_{\w_0}(\bar p_J)\mid 0\not=J\subsetneq [n]\}$, i.e. the Pl\"ucker coordinates form a Khovanskii basis for $(A_n,\val_{\w_0})$.
Moreover, we have
\[
Q_{\w_0}(\rho)= \conv (\val_{\w_0}(\bar p_J))_{  J\in\binom{[n]}{1}} + \dots +  \conv (\val_{\w_0}(\bar p_J))_{  J\in\binom{[n]}{n-1}}.
\]
\end{theorem}
\begin{proof}
We have $M_{\w_0}$ is of full rank and $\init_{\bw_{\w_0}}(I_n)=\init_{M_{\w_0}}(I_n)$. So, $\init_{M_{\w_0}}(I_n)$ is prime by assumption.
We apply Theorem~\ref{thm: val and quasi val with wt matrix} and deduce that $S(A_n,\val_{\w_0})$ is generated by $\{\val_{\w_0}(\bar p_J)\mid 0\not=J\subsetneq [n]\}$. Then
\[
\Delta(A_n,\val_{\w_0})=\conv (\val_{\w_0}(\bar p_J))_{  J\in\binom{[n]}{1}} + \dots +  \conv (\val_{\w_0}(\bar p_J))_{  J\in\binom{[n]}{n-1}}.
\]
By \cite[\S11]{FFL15} $Q_{\w_0}(\rho)=\Delta(A_n,\val_{\w_0})$, so the claim follows. 
\end{proof}

As a corollary of the above theorem, we prove an even stronger version of the \cite{BLMM}-conjecture.

\begin{corollary}\label{cor: prime implies MP}
Let $\w_0$ be a reduced expression of $w_0\in S_n$ and consider $\bw_{\w_0}\in\mathbb R^{K}$. Then
\[
\init_{\bw_{\w_0}}(I_n)\ \text{ is prime } \Leftrightarrow \ {\w_0} \ \text{ has the (strong) Minkowski property}.
\]
\end{corollary}
\begin{proof}
''$\Rightarrow$" 
By Theorem~\ref{thm: quasival for string}, $S(A_n,\val_{\w_0})$ is generated by $\{\val_{\w_0}(\bar p_{J})\}_{J\subsetneq [n]}$, which are ray generators of $C(A_n, \val_{\w_0})$.
Further, the generators are of form  $\val_{\w_0}(\bar p_{J})=(\omega_k,\cdot)$ for all $k\in[n-1]$ and $\vert J\vert=k$.
In particular, if $\lambda=\sum_{k=1}^{n-1} a_k\omega_k\in\Lambda^+$ we obtain by \eqref{eq:string and val}
\[
Q_{\w_0}(\lambda)= \textstyle \sum_{k=1}^{n-1} a_k  Q_{\w_0}(\omega_k).
\]

"$\Leftarrow$" Assume ${\w_0}$ has the strong Minkowski property. 
Then, by the reverse argument from above, $S(A_n,\val_{\w_0})$ is generated by $\{\val_{\w_0}(\bar p_J)\}_{J\subsetneq [n]}$.
Applying Theorem~\ref{thm: val and quasi val with wt matrix}, which we can do by the proof of Theorem~\ref{thm: quasival for string}, it follows that $\init_{{\bf w}_{\w_0}}(I_n)$ is prime.
\end{proof}

\begin{example}\label{exp:val vs valM}
Corollary~\ref{cor: prime implies MP} implies what we have seen from computations already, namely that $\w_0=s_1s_3s_2s_3s_1s_2\in$ String~4 does not satisfy MP.
The reason is that the element 
\[
\val_{\w_0}(\bar p_{2}\bar p_{134}+\bar p_{1}\bar p_{234})\succeq \min{}_{\prec}\{\val_{\w_0}(\bar p_2\bar p_{134}),\val_{\w_0}(\bar p_1\bar p_{234})\}
\]
is missing as a generator for $S(A_4,\val_{\w_0})$. As $\val_{\w_0}(\bar p_2\bar p_{134})=\val_{\w_0}(\bar p_1\bar p_{234})=(1,0,1,1,0,0)$, we deduce $\val_{\w_0}(\bar p_{2}\bar p_{134}+\bar p_{1}\bar p_{234})\succ (1,0,1,1,0)$.
Hence, this element can not be obtained from the images of Pl\"ucker coordinates under $\val_{\w_0}$ and therefore $\val_{\w_0}(\bar p_{2}\bar p_{134}+\bar p_{1}\bar p_{234})$ has to be added as a generator for $S(A_4,\val_{\w_0})$.
In this example, $\val$ and $\val_{M}$ are in fact different:
\[
\val_M(\bar p_{2}\bar p_{134}+\bar p_{1}\bar p_{234})=(1,0,1,1,0,0), \ \ \ \ \text{while} \ \ \ \ \val(\bar p_{2}\bar p_{134}+\bar p_{1}\bar p_{234})\succ (1,0,1,1,0,0)=:a.
\]
In particular, $\val_M$ does not have $1$-dimensional leaves as the quotient $F_{\val_M\succeq a}/F_{\val_M\succ a}$ is two-dimensional.
For $\val$ this is not the case, as $\overline{p_{2}p_{134}}=\overline{p_{1}p_{234}}$ in $\gr_\val(A_4)$.
\end{example}

\section{Application: Rietsch-Williams valuation from plabic graphs}\label{sec: exp plabic}

In this section we apply Theorem~\ref{thm: val and quasi val with wt matrix} from \S\ref{sec:val and quasival} to the valuation $\val_\G$ defined by Rietsch-Williams for Grassmannians using the cluster structure and Postnikov's plabic graphs in \cite{RW17}.
We identify a class of plabic graphs with non-integral associated Newton--Okounkov polytope.
The same combinatorial objects are used in \cite{BFFHL} to define weight vectors.
For $\Gr_2(\mathbb C^n)$ we show that these weight vectors yield the same toric degeneration of $\Gr_2(\mathbb C^n)$ as the corresponding Newton--Okounkov polytope.

Consider the Grassmannian with its Pl\"ucker embedding $\Gr_k(\mathbb C^n)\hookrightarrow \mathbb P^{\binom{n}{k}-1}$.
In this setting, its homogeneous coordinate ring $A_{k,n}$ is a quotient of the polynomial ring in Pl\"ucker variables $p_J$ with $J\subset [n]$ of cardinality $k$, denoted $J\in \binom{[n]}{k}$.
We quotient by the \emph{Pl\"ucker ideal} $I_{k,n}\subset \mathbb C[p_J\vert J\in \binom{[n]}{k}]$, a homogenous prime ideal generated by all \emph{Pl\"ucker relations}.
More precisely, for $K\in\binom{[n]}{k-1}$ and $L\in\binom{[n]}{k+1}$ let $\sgn(j;K,L):=(-1)^{\#\{l\in L\mid j<l \}+\#\{k\in K\mid k>j\}}$. 
Then the associated Pl\"ucker relation (see e.g. \cite[\S4.3]{M-S}) is of form
\begin{eqnarray}\label{eq: def plucker rel}
R_{K,L}:= \textstyle\sum_{j\in L} \sgn(j;K,L)p_{K\cup \{j\}}p_{L\setminus\{j\}}.
\end{eqnarray}
Then $A_{k,n}=\mathbb C[p_J]_J/I_{k,n}$. 
By \cite{Sco06}, $A_{k,n}$ has the structure of a \emph{cluster algebra} \cite{FZ02}: \emph{clusters} are sets of algebraically independent algebra generators of $A_{k,n}$ over an ambient field of rational functions. 
Together with certain combinatorial data (e.g. a quiver) a cluster forms a \emph{seed}.
They are related by \emph{mutation}, a procedure creating new seeds from a given one, that recovers the whole algebra after possibly infinitely many recursions.
In particular, certain subsets of \emph{Pl\"ucker coordinates} $\bar p_J\in A_{k,n}$ are clusters.
These special clusters are encoded by combinatorial objects called \emph{plabic graphs}\footnote{To be precise, we consider only reduced plabic graphs of trip permutation $\pi_{n-k,n}$, for details see Appendix~\ref{sec:plabic}.} \cite{Pos06}.
We recall plabic graphs in the Appendix~\ref{sec:plabic} below.

For every plabic graph $\mathcal G$ (or more generally every seed) for $\Gr_k(\mathbb C^n)$ in \cite{RW17} they define a valuation $\val_{\mathcal G}:A_{k,n}\setminus\{0\} \to \mathbb Z^{d}$ where $d:=k(n-k)=\dim \Gr_k(\mathbb C^n)$.
The images of Pl\"ucker coordinates $\val_\G(\bar p_J)$ can be computed using the combinatorics of the plabic graph $\mathcal G$.
Please consider Appendix~\ref{sec:appendix val} for the precise definition of the valuation and how to compute it. 
In what follows we use the terminology summarized there without further explanation. 

\begin{center}
\begin{figure}[h]
\centering
\begin{tikzpicture}[scale=.6]

        \begin{scope}[xshift=5cm,scale=.75]
       
\tikzset{->-/.style={decoration={
  markings,
  mark=at position #1 with {\arrow{>}}},postaction={decorate}}}
  \tikzset{-<-/.style={decoration={
  markings,
  mark=at position #1 with {\arrow{<}}},postaction={decorate}}}
    
\draw[->-=.5] (10,2) -- (9,2);    
\draw[->-=.5] (9.25,4) -- (6,4);
\draw[->-=.5] (6,4) -- (9,2);
\draw[->-=.5] (6,4) -- (7,2);
\draw[->-=.5] (6,4) -- (5,2);
\draw[->-=.5] (9,2) -- (8,1);
\draw[->-=.5] (8,1) -- (7,2);
\draw[->-=.5] (7,2) -- (6,1);
\draw[->-=.5] (6,1) -- (5,2);
\draw[->-=.5] (5,2) -- (5,-.25); 
\draw[->-=.5] (8,1) -- (8,-.85);
\draw[->-=.5] (6,1) -- (6,-.85);

\node[right] at (9.25,4) {1};
\node[right] at (10,2) {2};
\node[below] at (8,-.85) {3};
\node[below] at (6,-.85) {4};
\node[below] at (5,-.25) {5};

\draw (7,2) circle [radius=3];

\draw[fill, white] (6,4) circle [radius=.175];  
\draw (6,4) circle [radius=.175];  
\draw[fill, white] (8,1) circle [radius=.175];  
\draw (8,1) circle [radius=.175];  
\draw[fill, white] (6,1) circle [radius=.175];  
\draw (6,1) circle [radius=.175];  

\draw[fill] (9,2) circle [radius=.175]; 
\draw[fill] (7,2) circle [radius=.175]; 
\draw[fill] (5,2) circle [radius=.175];

\node at (8,2) {\scalebox{.35}{$\yng(2)$}};
\node at (6,2) {\scalebox{.35}{$\yng(1)$}};
\node at (9,3) {\scalebox{.35}{$\yng(3)$}};
\node at (9,1) {\scalebox{.35}{$\yng(3,3)$}};
\node at (7,.5) {\scalebox{.35}{$\yng(2,2)$}};
\node at (5.5,.5) {\scalebox{.35}{$\yng(1,1)$}};
\node at (5,3.5) {$\varnothing$};

        \end{scope}
\end{tikzpicture}
\caption{The plabic graph $\G^{\rm rec}$ of type $\pi_{3,5}$ with perfect orientation and source set $\{1,2\}$. Faces are labelled by Young tableaux as described in \S\ref{sec:plabic}.}
\label{fig:ExamplePerfectOrientation}
\end{figure}
\end{center}
\vspace{-1cm}

Let $\G$ be a reduced plabic graph for $\Gr_k(\mathbb C^n)$ with perfect orientation chosen such that $[k]$ is the source set, see e.g. Figure~\ref{fig:ExamplePerfectOrientation}.
Consider the weighting matrix $M_{\mathcal G}:=M_{\val_{\mathcal G}}$ of $\val_{\mathcal G}$ as defined in Appendix~\ref{sec:appendix val}.
The columns of $M_{\mathcal G}$ are $\val_\G(\bar p_{J})$ for $J\in\binom{[n]}{k}$ and the rows $M_1,\dots,M_{d+1}$ are indexed by the faces of the plabic graph $\mathcal G$.
Denote the boundary faces of $\G$ by $F_1,\dots,F_n$, where $F_i$ is adjacent to the boundary vertices $i$ and $i+1$.
Hence, $F_n=F_{\varnothing}$ is the face that does not contribute to the image of $\val_\G$ as $M_n=(0,\dots,0)$ and we omit this row of $M$.
Order the rows of $M_\G$ such that $M_i$ is the row corresponding to the face $F_i$ in $\G$.
In the following lemma we compute the columns of $M$ corresponding to boundary faces of $\G$ explicitly.

\begin{lemma}\label{lem: col corresp to bdy face}
Let $r\in[n-1]$ and $J=\{j_1,\dots,j_k\}\in\binom{[n]}{k}$ with $j_1<\dots j_s\le k<j_{s+1}<\dots <j_k$. Set $[k]\setminus\{j_1,\dots,j_s\}=\{i_1,\dots,i_{k-s}\}$ with $i_1<\dots<i_{k-s}$.
Then the $J^{\text{th}}$ entry of the column $M_r$ is
\[
(M_r)_J=\#\{l\mid r\in[i_l,j_{k-l+1}-1]\},
\]
where $[i_l,j_{k-l+1}]$ is the cyclic interval in $\mathbb Z/n\mathbb Z$.
\end{lemma}
\begin{proof}
Let $\bfff=\{\p_{j_1},\dots,\p_{j_k}\}$ be a flow (see Definition~\ref{def:flow}) from $[k]$ to $J$, where $\p_{j_i}$ denotes the path with sink $j_i$. 
The paths $\p_{j_r}$ for $r\le s$ are ``lazy paths", starting and ending at $j_r$ without moving.
Let $[k]\setminus \{j_1,\dots,j_s\}=\{i_1,\dots,i_{k-s}\}$ with $i_1<\dots<i_{k-s}$.
Hence, for $k-l+1>s$ we have non-trivial paths $\p_{j_{k-l+1}}$ with source $i_{l}$ and sink $j_{k-l+1}$.
To its left are all boundary faces $F_r$ with $r$ in the cyclic interval $[i_l,j_{k-l+1}-1]$. 
\end{proof}

Recall, the tropicalization of an ideal defined in \eqref{eq:trop}. The \emph{tropical Grassmannain}\cite{SS04} refers to the tropicalization of the Pl\"ucker ideal $I_{k,n}\subset \mathbb C[p_J]_J$, we denote it by $\trop(\Gr_k(\mathbb C^n))\subset \mathbb R^{\binom{n}{k}}$.
It contains an $n$-dimensional linear subspace $L_{I_{k,n}}:=\{w\in \mathbb R^{\binom{n}{k}}\mid \init_{w}(I_{k,n})=I_{k,n}\}$, called \emph{lineality space}, see \cite[page 393]{SS04}.
In the following corollary we treat columns of $M$ as weight vectors for Pl\"ucker variables. 

\begin{corollary}\label{cor: col in lin sp}
For all $r\in[n-1]$ we have $M_r\in L_{I_{k,n}}$.
\end{corollary}
\begin{proof}
Consider a Pl\"ucker relation $R_{K,L}$ with $K\in\binom{[n]}{k-1}$ and $L\in\binom{[n]}{k+1}$ of form \eqref{eq: def plucker rel}. 
Every term in $R_{K,L}$ equals $\pm p_Jp_{J'}$ for some $J,J'\in\binom{[n]}{k}$.
We can rewrite the formula for $(M_r)_J$ as follows
\[
(M_r)_J=\left\{ \begin{matrix} 
    \vert J\setminus [r-1]\vert, & \text{if } r<k,\\
    \vert J\cap [r+1,n]\vert , & \text{if } r\ge k.
\end{matrix}\right.
\]
Let $J=\{j_1,\dots,j_k\}$, $J'=\{j'_1,\dots,j'_k\}$ and define the sequence $S:=(j_1,\dots,j_k,j'_1,\dots,j'_k)$.
For any set $N$ we denote by $S\cap N$ the sequence obtained from $S$ when deleting all entries that are not elements of $N$.
Similarly, let $S\setminus N$ be the sequence obtained from $S$ when deleting all entries that do belong to $N$.
Further, for any sequence $S'$ let $\vert S'\vert$ denote its length.
Then 
\[
(M_r)_J+(M_r)_{J'}=\left\{
    \begin{matrix}
    \vert S\setminus [r-1]\vert, & \text{if } r<k,\\
    \vert S\cap [r+1,n]\vert , & \text{if } r\ge k.
    \end{matrix}
\right.
\]
Note that the right hand side only depends on the entries of $S$ regardless of the ordering.
Further, for all pairs $J,J'$ corresponding to monomials in $R_{K,L}$ the sequence $S$ is the same up to reordering.
Hence, $\init_{M_r}(R_{K,L})=R_{K,L}$ for all $r\in[n-1]$.
\end{proof}

Recall the plabic weight vector $\bw_\G$ from Definition~\ref{def: plabic deg}. The following proposition establishes the connection to what we have seen in \S\ref{sec:val and quasival}.
In terms of the weighting matrix $M_\G$, we observe
\[
\bw_\G= \textstyle \sum_{E_j \text{ interior face of }\G} M_j,
\]
where the sum contains exactly those $M_j$ corresponding to interior faces of $\G$.

\begin{table}[h]
    \centering
    \begin{tabular}{cc}
            \begin{tabular}{c|ccccccc}
        $\val_{\G^{\rm rec}}$ & $e_{35}$ & $e_{25}$ & $e_{45}$ & $e_{15}$ & $e_{12}$ & $e_{23}$ & $e_{34}$ \\ \hline
        $\bar p_{12}$  & 0 & 0 & 0 & 0 & 0 & 0 & 0 \\
        $\bar p_{13}$  & 0 & 0 & 0 & 0 & 1 & 0 & 0 \\
        $\bar p_{14}$  & 0 & 0 & 0 & 0 & 1 & 1 & 0 \\
        $\bar p_{15}$  & 0 & 0 & 0 & 0 & 1 & 1 & 1 \\ 
        $\bar p_{23}$  & 0 & 0 & 0 & 1 & 1 & 0 & 0 \\
        $\bar p_{24}$  & 0 & 0 & 0 & 1 & 1 & 1 & 0 \\
        $\bar p_{25}$  & 0 & 0 & 0 & 1 & 1 & 1 & 1 \\
        $\bar p_{34}$  & 0 & 1 & 0 & 1 & 2 & 1 & 0 \\
        $\bar p_{35}$  & 0 & 1 & 0 & 1 & 2 & 1 & 1 \\
        $\bar p_{45}$  & 1 & 1 & 0 & 1 & 2 & 2 & 1 
    \end{tabular} &  
    \begin{tabular}{c|ccccccc}
        ${\bf g}_{\G^{\rm rec}}$ & $f_{35}$ & $f_{25}$ & $f_{45}$ & $f_{15}$ & $f_{12}$ & $f_{23}$ & $f_{34}$ \\ \hline
        $\bar p_{12}$  & 0 & 0 & 0 & 0 & 1 & 0 & 0 \\
        $\bar p_{13}$  & 0 & -1 & 0 & 1 & 0 & 1 & 0 \\
        $\bar p_{14}$  & -1 & 0 & 0 & 1 & 0 & 0 & 1 \\
        $\bar p_{15}$  & 0 & 0 & 0 & 1 & 0 & 0 & 0 \\ 
        $\bar p_{23}$  & 0 & 0 & 0 & 0 & 0 & 1 & 0 \\
        $\bar p_{24}$  & -1 & 1 & 0 & 0 & 0 & 0 & 1 \\
        $\bar p_{25}$  & 0 & 1 & 0 & 0 & 0 & 0 & 0 \\
        $\bar p_{34}$  & 0 & 0 & 0 & 0 & 0 & 0 & 1 \\
        $\bar p_{35}$  & 1 & 0 & 0 & 0 & 0 & 0 & 0 \\
        $\bar p_{45}$  & 0 & 0 & 1 & 0 & 0 & 0 & 0 
    \end{tabular} 
    \end{tabular}

    \caption{The images of Pl\"ucker coordinates under the valuation $\val_{\G^{\rm rec}}$ for $\G^{\rm rec}$ as in Figure~\ref{fig:ExamplePerfectOrientation} on the left and under the valuation ${\bf g}_{\G^{\rm rec}}$ on the right. See \S\ref{sec:appendix val}\&\ref{sec:RW and GHKK} for details.}
    \label{tab:val 2-5}
\end{table}

\begin{example}\label{exp:plabic val}
Consider the plabic graph $\G=\mathcal{G}^{\rm rec}$ with perfect orientation from Figure \ref{fig:ExamplePerfectOrientation} and source set $[2]$. 
We compute $\deg_\G(p_J)$ and $\val_\G(\bar p_J)$ for all $J\in\binom{[5]}{2}$.
Order the faces of $\G$ by
\[
F_{35}=\scalebox{.35}{$\yng(1)$},\quad F_{25}=\scalebox{.35}{$\yng(2)$}, \quad F_{45}=\varnothing, \quad F_{15}=\scalebox{.35}{$\yng(3)$}, \quad F_{12}=\scalebox{.35}{$\yng(3,3)$}, \quad F_{23}=\scalebox{.35}{$\yng(2,2)$}\quad \text{and} \quad F_{34}=\scalebox{.35}{$\yng(1,1)$}.
\]
For example, consider $J=\{2,4\}$.
There are two flows, $\bfff_1$ and $\bfff_2$ from $[2]$ to $J=\{2,4\}$. Both consist of only one path from $1$ to $4$. 
One of them, say $\bfff_1$, has faces labelled by $\scalebox{.35}{$\yng(3)$},\scalebox{.35}{$\yng(3,3)$},\scalebox{.35}{$\yng(2,2)$}, \scalebox{.35}{$\yng(1,1)$}$ to its left while the other $\bfff_2$ has faces $\scalebox{.35}{$\yng(3)$},\scalebox{.35}{$\yng(3,3)$},\scalebox{.35}{$\yng(2,2)$},\scalebox{.35}{$\yng(1,1)$}$ and $\scalebox{.35}{\yng(2)}$ to its left. 
Then with respect to the above order of coordinates (corresponding to faces of $\G$) on $\mathbb Z^7$ we have
\[
\wei(\bfff_1)=(0,0,0,1,1,1,0) \ \text{ and } \ \wei(\bfff_2)=(0,1,0,1,1,1,0). 
\]
As $\deg_\G(\bfff_1)=0$ and $\deg_\G(\bfff_2)=1$, we have $\val_\G(\bar p_{24})=(1,0,1,1,0,0)$ and $\deg_\G(p_{24})=0$.
All other $\val_\G(\bar p_J)$ and $\deg_\G(p_J)$ can be recovered from Table~\ref{tab:val 2-5}.
\end{example}


\begin{figure}
    \centering

\begin{tikzpicture}[scale=.5]
\tikzset{->-/.style={decoration={
  markings,
  mark=at position #1 with {\arrow{>}}},postaction={decorate}}}
  \tikzset{-<-/.style={decoration={
  markings,
  mark=at position #1 with {\arrow{<}}},postaction={decorate}}}
    
\draw[->-=.5] (11,2) -- (9,2);    
\draw[->-=.5] (11,4) -- (6,4);
\draw[->-=.5] (6,4) -- (0,2);
\draw[->-=.5] (6,4) -- (9,2);
\draw[->-=.5] (6,4) -- (7,2);
\draw[->-=.5] (6,4) -- (5,2);
\draw[->-=.5] (6,4) -- (3,2);
\draw[->-=.5] (6,4) to [out=185,in=45] (-2,2);
\draw[->-=.5] (9,2) -- (8,1);
\draw[->-=.5] (8,1) -- (7,2);
\draw[->-=.5] (7,2) -- (6,1);
\draw[->-=.5] (6,1) -- (5,2);
\draw[->-=.5] (5,2) -- (4,1); 
\draw[->-=.5] (4,1) -- (3,2);
\draw[->-=.5] (8,1) -- (8,0);
\draw[->-=.5] (6,1) -- (6,0);
\draw[->-=.5] (4,1) -- (4,0);
\draw[->-=.5] (-1,1) -- (-1,0);
\draw[->-=.5] (0,2) -- (-1,1);
\draw[->-=.5] (-1,1) -- (-2,2);
\draw[->-=.5] (-2,2) -- (-3,0);
\draw[->-=.7] (3,2) -- (2.5,1.5);
\draw[->-=.5] (.5,1.5) -- (0,2);

\node at (1.5,1.5) {$\dots$};
\node[right] at (11,4) {1};
\node[right] at (11,2) {2};
\node[below] at (8,0) {3};
\node[below] at (6,0) {4};
\node[below] at (4,0) {5};
\node[below] at (-1,0) {$n-1$};
\node[below] at (-3,0) {$n$};
\node[below] at (1.5,0) {$\dots$};

\draw[fill, white] (6,4) circle [radius=.175];  
\draw (6,4) circle [radius=.175];  
\draw[fill, white] (8,1) circle [radius=.175];  
\draw (8,1) circle [radius=.175];  
\draw[fill, white] (6,1) circle [radius=.175];  
\draw (6,1) circle [radius=.175];  
\draw[fill, white] (4,1) circle [radius=.175];  
\draw (4,1) circle [radius=.175];  
\draw[fill, white] (-1,1) circle [radius=.175];  
\draw (-1,1) circle [radius=.175];  

\draw[fill] (9,2) circle [radius=.175]; 
\draw[fill] (7,2) circle [radius=.175]; 
\draw[fill] (5,2) circle [radius=.175]; 
\draw[fill] (3,2) circle [radius=.175]; 
\draw[fill] (0,2) circle [radius=.175]; 
\draw[fill] (-2,2) circle [radius=.175]; 

\node at (8,2) {\small $E_3$};
\node at (6,2) {\small $E_4$};
\node at (4,2) {\small $E_5$};
\node at (1,3) {\small $E_{n-1}$};
\node at (10,3) {\small $F_1$};
\node at (10,1) {\small $F_2$};
\node at (7,.5) {\small $F_3$};
\node at (5,.5) {\small $F_4$};
\node at (-1.875,.5) {\small $F_{n-1}$};
\end{tikzpicture}

    \caption{The plabic graph $\mathcal G^{\rm rec}$ for $\Gr_{2}(\mathbb C^n)$.}
    \label{fig:Grec2n}
\end{figure}

\begin{proposition}\label{prop: plabic lin form}
For every plabic graph $\G$ for $\Gr_2(\mathbb C^n)$ we have $\init_{M_\G}(I_{2,n})=\init_{\bw_G}(I_{2,n})$.
\end{proposition}

In general, for $k\ge 3$ and $n\ge 6$ the statement of the proposition is false, see Example~\ref{exp:3-6} and Theorem~\ref{thm:hexa not prime} below.

\begin{proof}
Fix the plabic graph $\mathcal G=\mathcal G^{\rm rec}$.
We show that for every Pl\"ucker relation $R_{i,j,k,l}:=p_{ij}p_{kl}-p_{ik}p_{jl}+p_{il}p_{jk}$ with $1\le i<j<k<l\le n$ we have $\init_{M_{\mathcal G}}(R_{ijkl})=\init_{{\bf w}_{\mathcal G}}(R_{ijkl})$.
As the $R_{ijkl}$ form a Gr\"obner basis for both initial ideals it follows $\init_{M_{\mathcal G}}(I_{2,n})=\init_{{\bf w}_{\mathcal G}}(I_{2,n})$.
For arbitrary plabic graphs the claim follows as both initial ideals change in the same way under the mutation move (M1). 
To see this compare \cite[\S13]{RW17} with \cite[\S6]{BFFHL} (or for more details \cite[\S3.3.4]{Thesis}). So without loss of generality we can focus on $\G^{\rm rec}$.

Consider $\mathcal G^{\rm rec}$ as in Figure~\ref{fig:Grec2n} and order the faces by $E_3,\dots,E_{n-1},F_1,\dots,F_n$.
By Corollary~\ref{cor: col in lin sp} we may focus only on the columns of $M_{\mathcal G^{\rm rec}}$ corresponding to the interior faces $E_3,\dots, E_{n-1}$.
We compute the leading monomials of the flow polynomials and the entries of ${\bf w}_{\mathcal G^{\rm rec}}$:
\begin{eqnarray*}
\begin{matrix}
\val_{\mathcal G^{\rm rec}}(\bar p_{1i}) &=& (0,\dots\dots\dots,0), &\quad {\bf w}_{\mathcal G^{\rm rec}}(\bar p_{1i}) &=& 0,\\ 
\val_{\mathcal G^{\rm rec}}(\bar p_{2i}) &=& (0,\dots\dots\dots,0), &\quad  {\bf w}_{\mathcal G^{\rm rec}}(\bar p_{2i}) &=& 0,\\
\val_{\mathcal G^{\rm rec}}(\bar p_{ij}) &=& (1,\dots,1,0,\dots,0), &\quad {\bf w}_{\mathcal G^{\rm rec}}(\bar p_{ij}) &=& i-2, 
\end{matrix}
\end{eqnarray*}
where $i<j$ and the $1$'s in $\val_{G^{\rm rec}}(\bar p_{ij})$ correspond to the faces $E_3,\dots,E_i$.
In particular, for $R_{ijkl}$ with $1\le i<j<k<l\le n$ we deduce $\init_{M_{\mathcal G^{\rm rec}}}(R_{ijkl})= -p_{ik}p_{jl} + p_{il}p_{jk} = \init_{{\bf w}_{\mathcal G^{\rm rec}}}(R_{ijkl})$.
\end{proof}

\begin{figure}[h]
    \centering
\begin{tikzpicture}[scale=.4]
\tikzset{->-/.style={decoration={
  markings,
  mark=at position #1 with {\arrow{>}}},postaction={decorate}}}
  \tikzset{-<-/.style={decoration={
  markings,
  mark=at position #1 with {\arrow{<}}},postaction={decorate}}}
    
\draw (0,0) circle [radius=5];  
  
\draw[-<-=.5] (-1,2) -- (1,2);
\draw[->-=.5] (1,2) -- (2.25,0);
\draw[->-=.5] (2.25,0) -- (1,-2);
\draw[->-=.5] (1,-2)-- (-1,-2);
\draw[-<-=.5] (-1,-2) -- (-2.25,0);
\draw[-<-=.5] (-2.25,0)-- (-1,2);
\draw[-<-=.5] (-1,2)  -- (-1,3.5);
\draw[->-=.5] (-1,3.5)-- (1,3.5);
\draw[->-=.5] (1,3.5) -- (1,2);
\draw[-<-=.5] (2.25,0) -- (3.5,-1);
\draw[->-=.7] (3.5,-1) -- (2.375,-3);
\draw[-<-=.5] (2.375,-3) -- (1,-2);
\draw[->-=.5] (-2.25,0) -- (-3.5,-1);
\draw[->-=.5] (-3.5,-1) -- (-2.375,-3);
\draw[-<-=.5] (-2.375,-3) -- (-1,-2);
\draw[->-=.5] (-1,4.9) -- (-1,3.5);
\draw[->-=.5] (1,4.9) -- (1,3.5);
\draw[->-=.5] (4.7,-1.75) -- (3.5,-1);
\draw[->-=.5] (2.375,-3) -- (3.35,-3.75);
\draw[->-=.5] (-3.5,-1) -- (-4.7,-1.75);
\draw[->-=.5] (-2.375,-3) -- (-3.35,-3.75);

\draw[fill] (1,3.5) circle [radius=.175];  
\draw[fill] (-1,2) circle [radius=.175];  
\draw[fill] (2.25,0) circle [radius=.175];  
\draw[fill] (2.375,-3) circle [radius=.175];  
\draw[fill] (-1,-2) circle [radius=.175];  
\draw[fill] (-3.5,-1) circle [radius=.175];  

\draw[fill, white] (-1,3.5) circle [radius=.175];  
\draw (-1,3.5) circle [radius=.175];  
\draw[fill, white] (1,2) circle [radius=.175];  
\draw (1,2) circle [radius=.175];  
\draw[fill, white] (3.5,-1) circle [radius=.175];  
\draw (3.5,-1) circle [radius=.175];
\draw[fill, white] (1,-2) circle [radius=.175];  
\draw (1,-2) circle [radius=.175];  
\draw[fill, white] (-2.25,0) circle [radius=.175];  
\draw (-2.25,0) circle [radius=.175];  
\draw[fill, white] (-2.375,-3) circle [radius=.175];  
\draw (-2.375,-3) circle [radius=.175];  

\node[above] at (-1,4.9) {1};
\node[above] at (1,4.9) {2};
\node[right] at (4.7,-1.75) {3};
\node[right] at (3.35,-3.75) {4};
\node[left] at (-4.7,-1.75) {6};
\node[left] at (-3.35,-3.75) {5};

\node at (0,0) {\scalebox{.3}{ $\yng(2,1)$}};
\node at (0,2.75) {\scalebox{.3}{$\yng(2)$}};
\node at (0,4.25) {\scalebox{.3}{$\yng(3)$}};
\node at (3,1.75) {\scalebox{.3}{$\yng(3,3)$}};
\node at (2.3,-1.5) {\scalebox{.3}{$\yng(3,3,1)$}};
\node at (3.7,-2.25) {\scalebox{.3}{$\yng(3,3,3)$}};
\node at (0,-3.5) {\scalebox{.3}{$\yng(2,2,2)$}};
\node at (-2.3,-1.5) {\scalebox{.3}{$\yng(1,1)$}};
\node at (-3.7,-2.25) {\scalebox{.3}{$\yng(1,1,1)$}};
\node at (-3,1.75) {{$\varnothing$}};

\begin{scope}[xshift=15cm]
\draw (0,0) circle [radius=5];  
  
\draw[-<-=.5] (-1,2) -- (1,2);
\draw[-<-=.5] (1,2) -- (2.25,0);
\draw[-<-=.5] (2.25,0) -- (1,-2);
\draw[->-=.5] (1,-2)-- (-1,-2);
\draw[-<-=.5] (-1,-2) -- (-2.25,0);
\draw[-<-=.5] (-2.25,0)-- (-1,2);
\draw[-<-=.5] (-1,2)  -- (-1,3.5);
\draw[-<-=.5] (-1,3.5)-- (1,3.5);
\draw[-<-=.5] (1,3.5) -- (1,2);
\draw[-<-=.5] (2.25,0) -- (3.5,-1);
\draw[->-=.7] (3.5,-1) -- (2.375,-3);
\draw[->-=.5] (2.375,-3) -- (1,-2);
\draw[->-=.5] (-2.25,0) -- (-3.5,-1);
\draw[-<-=.5] (-3.5,-1) -- (-2.375,-3);
\draw[-<-=.5] (-2.375,-3) -- (-1,-2);

\draw[-<-=.5] (-1,4.9) -- (-1,3.5);
\draw[->-=.5] (1,4.9) -- (1,3.5);
\draw[->-=.5] (4.7,-1.75) -- (3.5,-1);
\draw[-<-=.5] (2.375,-3) -- (3.35,-3.75);
\draw[->-=.5] (-3.5,-1) -- (-4.7,-1.75);
\draw[->-=.5] (-2.375,-3) -- (-3.35,-3.75);

\draw[fill] (1,3.5) circle [radius=.175];  
\draw[fill] (-1,2) circle [radius=.175];  
\draw[fill] (2.25,0) circle [radius=.175];  
\draw[fill] (2.375,-3) circle [radius=.175];  
\draw[fill] (-1,-2) circle [radius=.175];  
\draw[fill] (-3.5,-1) circle [radius=.175];  

\draw[fill, white] (-1,3.5) circle [radius=.175];  
\draw (-1,3.5) circle [radius=.175];  
\draw[fill, white] (1,2) circle [radius=.175];  
\draw (1,2) circle [radius=.175];  
\draw[fill, white] (3.5,-1) circle [radius=.175];  
\draw (3.5,-1) circle [radius=.175];
\draw[fill, white] (1,-2) circle [radius=.175];  
\draw (1,-2) circle [radius=.175];  
\draw[fill, white] (-2.25,0) circle [radius=.175];  
\draw (-2.25,0) circle [radius=.175];  
\draw[fill, white] (-2.375,-3) circle [radius=.175];  
\draw (-2.375,-3) circle [radius=.175];  

\node[above] at (-1,4.9) {6};
\node[above] at (1,4.9) {1};
\node[right] at (4.7,-1.75) {2};
\node[right] at (3.35,-3.75) {3};
\node[left] at (-4.7,-1.75) {5};
\node[left] at (-3.35,-3.75) {4};

\node at (0,0) {\scalebox{.3}{ $\yng(3,2,1)$}};
\node at (0,2.75) {\scalebox{.3}{$\yng(3,1,1)$}};
\node at (0,4.25) {{$\varnothing$}};
\node at (3,1.75) {\scalebox{.3}{$\yng(3)$}};
\node at (2.3,-1.5) {\scalebox{.3}{$\yng(3,2)$}};
\node at (3.7,-2.25) {\scalebox{.3}{$\yng(3,3)$}};
\node at (0,-3.5) {\scalebox{.3}{$\yng(3,3,3)$}};
\node at (-2.3,-1.5) {\scalebox{.3}{$\yng(2,2,1)$}};
\node at (-3.7,-2.25) {\scalebox{.3}{$\yng(2,2,2)$}};
\node at (-3,1.75) {\scalebox{.3}{$\yng(1,1,1)$}};

\end{scope}

\end{tikzpicture}
    \caption{Two plabic graphs for $\Gr_3(\mathbb C^6)$ for which $\init_{M_{\mathcal G}}(I_{3,6})$ is not prime (see \cite[\S9]{RW17}).}
    \label{fig:3-6}
\end{figure}

\begin{example}\label{exp:3-6}
For $\Gr_3(\mathbb C^6)$ we consider the plabic graph $\mathcal G$ on the left in Figure~\ref{fig:3-6} and compute the images of Pl\"ucker coordinates under $\val_{\mathcal G}:A_{3,6}\setminus \{0\}\to \mathbb Z^{9}$ (see Table~\ref{tab:3-6} below).
With respect to the order on variables as indicated in Table~\ref{tab:3-6} we fix the lexicographic order on $\mathbb Z^9$. 

Take the following four-term Pl\"ucker relation written as a sum of two binomials:
\[
(p_{123}p_{456} - p_{124}p_{356}) + (p_{125}p_{346} - p_{126}p_{345}) =: f_1+f_2 
\]
We compute the flow polynomials for $\bar f_1$ and $\bar f_2$ and obtain:
$\val_{\mathcal G}(\bar f_1)=\val_{\mathcal G}(\bar f_2) = (1,2,3,2,1,1,1,3,1)$.
Comparing with the valuations of Pl\"ucker coordinates we deduce $\val_{\mathcal G}(\bar f_i)$ does not lie in the semigroup-span of $\{\val_{\mathcal G}(\bar p_{ijk})\mid 1\le i<j<k\le 6\}$. Hence, by Theorem~\ref{thm: val and quasi val with wt matrix} the associated initial ideal $\init_{M_{\mathcal G}}(I_{3,6})$ is not prime.
Further, computing initial forms of $f_1+f_2$ we see $\init_{{\bf w}_{\mathcal G}}(f_1+f_2)=f_1+f_2$, but $\init_{M_{\mathcal G}}(f_1+f_2) = f_1$.
So, $\init_{M_{\mathcal G}}(I_{3,6})\not = \init_{{\bf w}_{\mathcal G}}(I_{3,6})$.
For the plabic graph on the right side of Figure~\ref{fig:3-6} a similar argument works considering the Pl\"ucker relation $(p_{123}p_{456}-p_{145}p_{236})+(p_{146}p_{235}-p_{156}p_{234})$.
\end{example}

\begin{definition}\label{def:hexagonal}
We say a plabic graph $\mathcal G$ is \emph{hexagonal}, if locally around  six consecutive boundary vertices it is equivalent up to moves (M2) and (M3) to the arrangement depicted on the left in Figure~\ref{fig:hexagon orient}.
\end{definition}

\begin{remark}
We call such plabic graphs hexagonal in alignment with \cite{MoSh} where a criterion for non-prime cones in $\trop(\Gr_3(\mathbb C^n))$ is given in terms of \emph{hexagonal} tropical line arrangement. We believe that a combinatorial algorithm transforming hexagonal plabic graphs for $\Gr_3(\mathbb C^n)$ into hexagonal tropical line arrangements should exist.
\end{remark}

\begin{figure}[h]
 \centering
\begin{tikzpicture}[scale=.35]
\tikzset{->-/.style={decoration={
  markings,
  mark=at position #1 with {\arrow{>}}},postaction={decorate}}}
  \tikzset{-<-/.style={decoration={
  markings,
  mark=at position #1 with {\arrow{<}}},postaction={decorate}}}

\draw (-1,2) -- (1,2);
\draw (1,2) -- (2.25,0);
\draw (2.25,0) -- (1,-2);
\draw (1,-2)-- (-1,-2);
\draw (-1,-2) -- (-2.25,0);
\draw (-2.25,0)-- (-1,2);
\draw (-1,2)  -- (-1,3.5);
\draw (-1,3.5)-- (1,3.5);
\draw (1,3.5) -- (1,2);
\draw (2.25,0) -- (3.5,-1);
\draw (3.5,-1) -- (2.375,-3);
\draw (2.375,-3) -- (1,-2);
\draw (-2.25,0) -- (-3.5,-1);
\draw (-3.5,-1) -- (-2.375,-3);
\draw (-2.375,-3) -- (-1,-2);
\draw (-1.5,5) -- (-1,3.5);
\draw (1,5) -- (1,3.5);
\draw (4.7,-1.75) -- (3.5,-1);
\draw (2.375,-3) -- (3.35,-3.75);
\draw (-3.5,-1) -- (-6.7,-3.75);
\draw (-2.375,-3) -- (-3.35,-3.9);

\draw (-1,3.5) -- (-2,4.25);
\draw (-1,3.5) -- (-2.5,3.5);
\draw (-1,3.5) -- (-2,2.75);
\node at (-4,2.5) {$\dots$};
\draw (-3.5,-1) -- (-4,1);
\draw (-3.5,-1) -- (-4.5,0);
\draw (-3.5,-1) -- (-4.5,-1);

\draw (-5,5) -- (1,5) to  [out=0,in=80]  (4.7,-1.75) to [out=260,in=40] (3.35,-3.75) to [out=190,in=0] (-7.25,-3.75);

\draw[fill] (1,3.5) circle [radius=.175];  
\draw[fill] (-1,2) circle [radius=.175];  
\draw[fill] (2.25,0) circle [radius=.175];  
\draw[fill] (2.375,-3) circle [radius=.175];  
\draw[fill] (-1,-2) circle [radius=.175];  
\draw[fill] (-3.5,-1) circle [radius=.175];  

\draw[fill, white] (-1,3.5) circle [radius=.175];  
\draw (-1,3.5) circle [radius=.175];  
\draw[fill, white] (1,2) circle [radius=.175];  
\draw (1,2) circle [radius=.175];  
\draw[fill, white] (3.5,-1) circle [radius=.175];  
\draw (3.5,-1) circle [radius=.175];
\draw[fill, white] (1,-2) circle [radius=.175];  
\draw (1,-2) circle [radius=.175];  
\draw[fill, white] (-2.25,0) circle [radius=.175];  
\draw (-2.25,0) circle [radius=.175];  
\draw[fill, white] (-2.375,-3) circle [radius=.175];  
\draw (-2.375,-3) circle [radius=.175];  

\begin{scope}[xshift=20cm]
\draw[-<-=.5] (-1,2) -- (1,2);
\draw[->-=.5] (1,2) -- (2.25,0);
\draw[->-=.5] (2.25,0) -- (1,-2);
\draw[->-=.5] (1,-2)-- (-1,-2);
\draw[-<-=.5] (-1,-2) -- (-2.25,0);
\draw[-<-=.5] (-2.25,0)-- (-1,2);
\draw[-<-=.5] (-1,2)  -- (-1,3.5);
\draw[->-=.5] (-1,3.5)-- (1,3.5);
\draw[->-=.5] (1,3.5) -- (1,2);
\draw[-<-=.5] (2.25,0) -- (3.5,-1);
\draw[->-=.7] (3.5,-1) -- (2.375,-3);
\draw[-<-=.5] (2.375,-3) -- (1,-2);
\draw[->-=.5] (-2.25,0) -- (-3.5,-1);
\draw[-<-=.5] (-3.5,-1) -- (-2.375,-3);
\draw[-<-=.5] (-2.375,-3) -- (-1,-2);
\draw[->-=.5] (-1.5,5) -- (-1,3.5);
\draw[->-=.5] (1,5) -- (1,3.5);
\draw[->-=.5] (4.7,-1.75) -- (3.5,-1);
\draw[->-=.5] (2.375,-3) -- (3.35,-3.75);
\draw[->-=.5] (-3.5,-1) -- (-6.7,-3.75);
\draw[->-=.5] (-2.375,-3) -- (-3.35,-3.9);

\draw (-1,3.5) -- (-2,4.25);
\draw (-1,3.5) -- (-2.5,3.5);
\draw (-1,3.5) -- (-2,2.75);
\node at (-4,2.5) {$\dots$};
\draw (-3.5,-1) -- (-4,1);
\draw (-3.5,-1) -- (-4.5,0);
\draw (-3.5,-1) -- (-4.5,-1);

\draw (-5,5) -- (1,5) to  [out=0,in=80]  (4.7,-1.75) to [out=260,in=40] (3.35,-3.75) to [out=190,in=0] (-7.25,-3.75);         

\draw[fill] (1,3.5) circle [radius=.175];  
\draw[fill] (-1,2) circle [radius=.175];  
\draw[fill] (2.25,0) circle [radius=.175];  
\draw[fill] (2.375,-3) circle [radius=.175];  
\draw[fill] (-1,-2) circle [radius=.175];  
\draw[fill] (-3.5,-1) circle [radius=.175];  

\draw[fill, white] (-1,3.5) circle [radius=.175];  
\draw (-1,3.5) circle [radius=.175];  
\draw[fill, white] (1,2) circle [radius=.175];  
\draw (1,2) circle [radius=.175];  
\draw[fill, white] (3.5,-1) circle [radius=.175];  
\draw (3.5,-1) circle [radius=.175];
\draw[fill, white] (1,-2) circle [radius=.175];  
\draw (1,-2) circle [radius=.175];  
\draw[fill, white] (-2.25,0) circle [radius=.175];  
\draw (-2.25,0) circle [radius=.175];  
\draw[fill, white] (-2.375,-3) circle [radius=.175];  
\draw (-2.375,-3) circle [radius=.175];  

\node at (0,0) {\footnotesize{$B$}};
\node at (0,4.25) {\footnotesize{$F_{k-2}$}};
\node at (0,2.75) {\footnotesize{A}};
\node at (3,1.75) {\footnotesize{$F_{k-1}$}};
\node at (2.3,-1.5) {\footnotesize{$C$}};
\node at (3.7,-2.25) {\footnotesize{$F_{k}$}};
\node at (0,-3.5) {\footnotesize{$F_{k+1}$}};
\node at (-2.3,-1.5) {\footnotesize{$D$}};
\node at (-4.25,-3) {\footnotesize{$F_{k+2}$}};

\node[above] at (-1.5,5) {\footnotesize{$k-2$}};    
\node[above] at (1.15,5) {\footnotesize{$k-1$}};
\node[right] at (4.7,-1.75) {\footnotesize{$k$}};
\node[below] at (3.35,-3.75) {\footnotesize{$k+1$}};
\node[below] at (-3.35,-3.85) {\footnotesize{$k+2$}};
\node[below] at (-6.7,-3.75) {\footnotesize{$k+3$}};

\end{scope}
\end{tikzpicture}
    \caption{On the left: the local structure of a hexagonal plabic graph. On the right: a possible labelling with perfect orientation of a hexagonal plabic graph.}
    \label{fig:hexagon orient}
\end{figure}

\begin{theorem}\label{thm:hexa not prime}
Let $\mathcal G$ be a hexagonal plabic graph. Then $\init_{M_{\mathcal G}}(I_{k,n})$ is not prime. Moreover, $\init_{M_{\mathcal G}}(I_{k,n})\not =\init_{{\bf w}_{\mathcal G}}(I_{k,n})$ and $\Delta(A_{k,n},\val_{\mathcal G})$ is not integral.
\end{theorem}

\begin{proof}
We will show that the value semigroup $S(A_{k,n},\val_{\mathcal G})$ associated with the valuation of a plabic graph as on the right side of Figure~\ref{fig:hexagon orient} is not generated by $\{\val_{\mathcal G}(\bar p_J)\mid J\in \binom{[n]}{k}\}$ and how this is enough to deduce the claims of the Theorem.
Assuming this statement is true by the isomorphism of \cite{BCMN} (see \S\ref{sec:RW and GHKK}) the same is true for the valuation ${\bf g}_{\mathcal G}:A_{k,n}\setminus \{0\}\to \mathbb Z^{d}$ sending every Pl\"ucker coordinate to its ${\bf g}$-vector (see \S\ref{sec:RW and GHKK}).
While $\val_{\mathcal G}$ depends on the perfect orientation given to $\mathcal G$, the valuation ${\bf g}_{\mathcal G}$ purely depends on the combinatorial type of $\mathcal G$ (on the associated quiver to be precise). 
If $\mathcal G'$ is a hexagonal plabic graph obtained from $\mathcal G$ by permuting the labelling of the boundary vertices this induces a natural isomorphism between $S(A_{k,n},{\bf g}_{\mathcal G'})\cong S(A_{k,n},{\bf g}_{\mathcal G})$. Combining with \cite{BCMN} we obtain
\[
S(A_{k,n}, \val_{\mathcal G'})\cong S(A_{k,n},{\bf g}_{\mathcal G'})\cong S(A_{k,n},{\bf g}_{\mathcal G})\cong S(A_{k,n}, \val_{\mathcal G}).
\]

So without loss of generality that $\mathcal G$ is perfectly oriented and of form as on the right in Figure~\ref{fig:hexagon orient}. 
Consider the Pl\"ucker relation
\begin{eqnarray*}
f_1+f_2 &:=& (p_{[k]}p_{[k-3]\cup \{k+1,k+2,k+3\}} - p_{[k-1]\cup \{k+1\}}p_{[k-3]\cup \{k,k+2,k+3\}}) \\
&+& (p_{[k-1]\cup\{k+2\}]}p_{[k-3]\cup\{k,k+1,k+3\}} - p_{[k-1]\cup\{k+3\}}p_{[k-3]\cup \{k,k+1,k+2\}}).
\end{eqnarray*}

\underline{\emph{Claim:}} The image of $\bar f_1$ under $\val_{\mathcal G}$ does not lie in the semigroup span of $\{\val_{\mathcal G}(\bar p_J)\mid J\in\binom{[n]}{k}\}$.

\smallskip 
\noindent 
We compute (the lowest degree terms of) the flow polynomials
\begin{eqnarray*}
\bar p_{[k]} &=& 1, \\
\bar p_{[k-1]\cup \{k+1\}} &=& x_{F_k}(1+x_C), \\
\bar p_{[k-1]\cup \{k+2\}} &=& x_C x_{F_k}x_{F_{k+1}}, \\
\bar p_{[k-1]\cup \{k+3\}} &=& x_Cx_{F_k}x_{F_{k+1}}x_{F_{k+2}},\\
\bar p_{[k-3]\cup \{k,k+1,k+2\}} &=& x_Ax_Bx_C^2x_{F_{k-2}}x_{F_{k-1}}^2x_{F_k}^2x_{F_{k+1}},\\
\bar p_{[k-3]\cup \{k,k+1,k+3\}} &=& x_Ax_Bx_C^2x_{F_{k-2}}x_{F_{k-1}}^2x_{F_k}^2x_{F_{k+1}}x_{F_{k+2}}(1+x_D+ {\rm h.o.t.}), \\
\bar p_{[k-3]\cup \{k,k+2,k+3\}} &=& x_Ax_Bx_C^2x_D x_{F_{k-2}}x_{F_{k-1}}^2x_{F_k}^2x_{F_{k+1}}^2x_{F_{k+2}} (1+ {\rm h.o.t.}), \\
\bar p_{[k-3]\cup \{k+1,k+2,k+3\}} &=&  x_Ax_Bx_C^2x_Dx_{F_{k-2}}x_{F_{k-1}}^2x_{F_k}^3x_{F_{k+1}}^2x_{F_{k+2}}.
\end{eqnarray*}

Ordering the variables corresponding to faces of $\mathcal G$ by $A,B,C,D,F_{k-2},F_{k-1},F_{k},F_{k+1},F_{k+2}$ followed by the ones not displayed in Figure~\ref{fig:hexagon orient} we obtain
\begin{eqnarray*}
\val_{\mathcal G}(\bar p_{[k]}\bar p_{[k-3]\cup \{k+1,k+2,k+3\}}) &=& (1,1,2,1,1,2,3,2,1,0,\dots,0) ,\\
\val_{\mathcal G}(\bar p_{[k-1]\cup \{k+1\}}\bar p_{[k-3]\cup \{k,k+2,k+3\}}) &=& (1,1,2,1,1,2,3,2,1,0,\dots,0) ,\\
\val_{\mathcal G}(\bar p_{[k-1]\cup\{k+2\}]}\bar p_{[k-3]\cup\{k,k+1,k+3\}}) &=& (1,1,3,0,1,2,3,2,1,0,\dots,0),\\
\val_{\mathcal G}(\bar p_{[k-1]\cup\{k+3\}}\bar p_{[k-3]\cup \{k,k+1,k+2\}}) &=& (1,1,3,1,1,2,3,2,1,0,\dots,0 ),\\
\val_{\mathcal G}(\bar f_1)=\val_{\mathcal G}(\bar f_2) &=& (1,1,3,1,1,2,3,2,1,0,\dots,0).
\end{eqnarray*}

If $\val_{\mathcal G}(\bar f_1)$ would lie in the semigroup span of the values of Pl\"ucker coordinates,
by Corollary~\ref{cor: col in lin sp} it has to be of the form
\[
\val_{\mathcal G}(\bar f_1)= \val_{\mathcal G}(\bar p_{I_1})+\val_{\mathcal G}(\bar p_{I_2}),
\]
for $I_1,I_2\in \binom{[n]}{k}$ such that $I_1\cap I_2=[k-3]$ and $I_1\cup I_2=[k+3]$.
The values of such Pl\"ucker coordinates resemble the ones in the case of $\Gr_3(\mathbb C^6)$ from Table~\ref{tab:3-6}. This can be seem by making the following identifications:
identify faces of the plabic graph $\mathcal G$ and the one on the left in Figure~\ref{fig:3-6}: 
\[
\scalebox{.25}{\yng(2)}\longleftrightarrow A, \ \ 
\scalebox{.25}{\yng(2,1)} \longleftrightarrow B, \ \ 
\scalebox{.25}{\yng(3,3,1)}\longleftrightarrow C, \ \
\scalebox{.25}{\yng(1,1)} \longleftrightarrow D, \ \  \scalebox{.25}{\yng(3)} \longleftrightarrow F_{k-2},
\dots, \scalebox{.25}{\yng(1,1,1)}\longleftrightarrow F_{k+2},
\]
and identify Pl\"ucker coordinates $\bar p_{i,j,l} \longleftrightarrow \bar p_{[k-3]\cup \{i+k-3,j+k-3,l+k-3\}}$ for $1\le i<j<l\le 6$. As in Example~\ref{exp:3-6} we see that the value $(1,1,3,1,1,2,3,2,1,\dots)$ does not arise as a sum of any two corresponding values in Table~\ref{tab:3-6}, so the claim follows.

Lemma~\ref{lem:gv identity} and Theorem(\cite{BCMN}) stated in \S\ref{sec:RW and GHKK} imply that $M_{\mathcal G}$ is of full rank.
Then by Theorem~\ref{thm: val and quasi val with wt matrix} it follows that $\init_{M_{\mathcal G}}(I_{k,n})$ is not prime. 
Further, we have $\init_{M_{\mathcal G}}(f_1+f_2)=f_1$ and $\init_{{\bf w}_{\mathcal G}}(f_1+f_2)=f_1+f_2$.
Hence, $\init_{M_{\mathcal G}}(I_{k,n})\not =\init_{{\bf w}_{\mathcal G}}(I_{k,n})$.
Lastly, the ray spanned by $\val_{\mathcal G}(\bar f_1)$ gives a vertex of $\Delta(A_{k,n},\val_{\mathcal G})$ that is of form
\[
\left(\frac{1}{2},\frac{1}{2},\frac{3}{2},\frac{1}{2},\frac{1}{2},1,\frac{3}{2},1,\frac{1}{2},0,\dots, 0\right),
\]
which makes $\Delta(A_{k,n},\val_{\mathcal G})$ non-integral.
\end{proof}

\begin{table}[]
    \centering
\begin{tabular}{ll}    
    \begin{tabular}{c|c c c c c c c c c c}
    & \scalebox{.2}{$\yng(3)$} & \scalebox{.2}{$\yng(3,3)$} & \scalebox{.2}{$\yng(3,3,3)$} & \scalebox{.2}{$\yng(2,2,2)$} & \scalebox{.2}{$\yng(1,1,1)$} &  & \scalebox{.2}{$\yng(2)$} & \scalebox{.2}{$\yng(2,1)$} & \scalebox{.2}{$\yng(3,3,1)$} & \scalebox{.2}{$\yng(1,1)$} \\ \hline
    $\bar p_{123} $  & 0 & 0 & 0 & 0 & 0 & & 0 & 0 & 0 & 0 \\
    $\bar p_{124} $  & 0 & 0 & 1 & 0 & 0 & & 0 & 0 & 0 & 0 \\
    $\bar p_{125} $  & 0 & 0 & 1 & 1 & 0 & & 0 & 0 & 1 & 0 \\
    $\bar p_{126} $  & 0 & 0 & 1 & 1 & 1 & & 0 & 0 & 1 & 0 \\
    $\bar p_{134} $  & 0 & 1 & 1 & 0 & 0 & & 0 & 0 & 1 & 0 \\
    $\bar p_{135} $  & 0 & 1 & 1 & 1 & 0 & & 0 & 0 & 1 & 0 \\
    $\bar p_{136} $  & 0 & 1 & 1 & 1 & 1 & & 0 & 0 & 1 & 0 \\
    $\bar p_{145} $  & 0 & 1 & 2 & 1 & 0 & & 0 & 0 & 1 & 0 \\
    $\bar p_{146} $  & 0 & 1 & 2 & 1 & 1 & & 0 & 0 & 1 & 0 \\
    $\bar p_{156} $  & 0 & 1 & 2 & 2 & 0 & & 0 & 1 & 2 & 1 \\
    \end{tabular}
&
    \begin{tabular}{c|c c c c c c c c c c}
     & \scalebox{.2}{$\yng(3)$} & \scalebox{.2}{$\yng(3,3)$} & \scalebox{.2}{$\yng(3,3,3)$} & \scalebox{.2}{$\yng(2,2,2)$} & \scalebox{.2}{$\yng(1,1,1)$} &  & \scalebox{.2}{$\yng(2)$} & \scalebox{.2}{$\yng(2,1)$} & \scalebox{.2}{$\yng(3,3,1)$} & \scalebox{.2}{$\yng(1,1)$} \\ \hline    
    $\bar p_{234} $  & 1 & 1 & 1 & 0 & 0 & & 0 & 0 & 1 & 0 \\
    $\bar p_{235} $  & 1 & 1 & 1 & 1 & 0 & & 0 & 0 & 1 & 0 \\
    $\bar p_{236} $  & 1 & 1 & 1 & 1 & 1 & & 0 & 0 & 1 & 0 \\
    $\bar p_{245} $  & 1 & 1 & 2 & 1 & 0 & & 0 & 0 & 1 & 0 \\
    $\bar p_{246} $  & 1 & 1 & 2 & 1 & 1 & & 0 & 0 & 1 & 0 \\
    $\bar p_{256} $  & 1 & 1 & 2 & 2 & 1 & & 0 & 1 & 2 & 1 \\
    $\bar p_{345} $  & 1 & 2 & 2 & 1 & 0 & & 1 & 1 & 2 & 0 \\
    $\bar p_{346} $  & 1 & 2 & 2 & 1 & 1 & & 1 & 1 & 2 & 0 \\
    $\bar p_{356} $  & 1 & 2 & 2 & 2 & 1 & & 1 & 1 & 2 & 1 \\
    $\bar p_{456} $  & 1 & 2 & 3 & 2 & 1 & & 1 & 1 & 2 & 1 \\
    \end{tabular}
\end{tabular}    
    \caption{The images of Pl\"ucker coordinates under the valuation $\val_{\mathcal G}$ for the plabic graph $\mathcal G$ as on the left in Figure~\ref{fig:3-6}.}
    \label{tab:3-6}
\end{table}

\begin{appendices}
\section{Appendix for \S\ref{sec: exp plabic}}
\subsection{Plabic graphs}\label{sec:plabic}

We review the definition of plabic graphs due to Postnikov \cite{Pos06}. This section is closely oriented towards \cite{RW17} and \cite{BFFHL}.

\begin{definition}\label{def: plabic graph}
A \textit{plabic graph} $\mathcal{G}$ is a planar bicolored graph embedded in a disk (up to homotopy). It has $n$ boundary vertices numbered $1,\dots, n$ in a clockwise order. Boundary vertices lie on the boundary of the disk and are not coloured. Additionally, there are internal vertices coloured black or white. Each boundary vertex is adjacent to a single internal vertex.
\end{definition}

For our purposes we assume that plabic graphs are connected and that every leaf of a plabic graph is a boundary vertex. We first recall the four local moves on plabic graphs.

\begin{itemize}
\item[(M1)] If a plabic graph contains a square of four internal vertices with alternating colours, each of which is trivalent, then the colours can be swapped. So every black vertex in the square becomes white and every white vertex becomes black (see Figure~\ref{fig: plabic move M12}).

\item[(M2)] If two internal vertices of the same colour are connected by an edge, the edge can be contracted and the two vertices can be merged. Conversely, any internal black or white vertex can be split into two adjacent vertices of the same colour (see Figure~\ref{fig: plabic move M12}).

\begin{center}
    \begin{figure}[h]
    \centering
    \begin{tikzpicture}[scale=0.5]
    \draw (0,0) -- (1,1) -- (2,1) -- (3,0);
    \draw (0,3) -- (1,2) -- (2,2) -- (3,3);
    \draw (1,2) -- (1,1);
    \draw (2,2) -- (2,1);

    \draw[->] (3.5,1.5) -- (4.5,1.5);
    \draw[->] (4.5,1.5) -- (3.5,1.5);

    \draw (5,0) -- (6,1) -- (7,1) -- (8,0);
    \draw (5,3) -- (6,2) -- (7,2) -- (8,3);
    \draw (6,2) -- (6,1);
    \draw (7,2) -- (7,1);

    \draw[fill] (1,1) circle [radius=0.1];
    \draw[fill] (2,2) circle [radius=0.1];
    \draw[fill, white] (1,2) circle [radius=0.1];
    \draw[fill, white] (2,1) circle [radius=0.1];
    \draw (1,2) circle [radius=0.1];
    \draw (2,1) circle [radius=0.1];

    \draw[fill] (6,2) circle [radius=0.1];
    \draw[fill] (7,1) circle [radius=0.1];
    \draw[fill, white] (6,1) circle [radius=0.1];
    \draw[fill, white] (7,2) circle [radius=0.1];
    \draw (6,1) circle [radius=0.1];
    \draw (7,2) circle [radius=0.1];
    
    \begin{scope}[xshift=12cm]
     \draw (0,0.5) -- (1,1) -- (0,1.5);
    \draw (1,1) -- (2,1) -- (3,2);
    \draw (3,1) -- (2,1) -- (3,0);

    \draw[->] (3.5,1) -- (4.5,1);
    \draw[->] (4.5,1) -- (3.5,1);

    \draw (5,0.5) -- (6,1) -- (5,1.5);
    \draw (7,2) -- (6,1) -- (7,1);
    \draw (6,1) -- (7,0);

    \draw[fill] (6,1) circle [radius=0.1];
    \draw[fill] (1,1) circle [radius=0.1];
    \draw[fill] (2,1) circle [radius=0.1];

    \end{scope}
    \end{tikzpicture}
    \caption{Square move (M1), and merging vertices of same colour (M2)}
     \label{fig: plabic move M12}
    \end{figure}
    \end{center}

\vspace{-1cm}
\item[(M3)] If a plabic graph contains an internal vertex of degree $2$, it can be removed. Equivalently, an internal black or white vertex can be inserted in the middle of any edge (see Figure~\ref{fig: plabic move M3R}).

\item[(R)] If two internal vertices of opposite colour are connected by two parallel edges, they can be reduced to only one edge. This can not be done conversely (see Figure~\ref{fig: plabic move M3R}).
\end{itemize}
    \begin{center}
    \begin{figure}[h]
    \centering
    \begin{tikzpicture}[scale=0.7]
    \draw (0,0) -- (1,0) -- (2,0);

    \draw[->] (2.5,0) -- (3.5,0);
    \draw[->] (3.5,0) -- (2.5,0);

    \draw (4,0) -- (6,0);

    \draw[fill, white] (1,0) circle [radius=0.1];
    \draw (1,0) circle [radius=0.1];
    
    \begin{scope}[xshift=10cm]
       \draw (0,0) -- (1,0);
    \draw (2,0) -- (3,0);

    \draw (1,0.06) -- (2,0.06);
    \draw (1,-.06) -- (2,-.06);

    \draw[fill] (1,0) circle [radius=0.1];
    \draw[fill, white] (2,0) circle [radius=0.1];
    \draw (2,0) circle [radius=0.1];

    \draw[->] (3.5,0) -- (4.5,0);

    \draw (5,0) -- (7,0);

    \end{scope}
    
    \end{tikzpicture}
    \caption{Insert/remove degree two vertex (M3), and reducing parallel edges (R)}
    \label{fig: plabic move M3R}
    \end{figure}
    \end{center}
\vspace{-.75cm}
The \textit{equivalence class} of a plabic graph $\mathcal G$ is defined as the set of all plabic graphs that can be obtained from $\mathcal G$ by applying (M1)-(M3). 
If in the equivalence class there is no graph to which (R) can be applied, we say $\mathcal G$ is \textit{reduced}. 
From now on we only consider reduced plabic graphs.

\begin{definition}\label{def: trip permutation}
Let $\mathcal G$ be a reduced plabic graph with boundary vertices $v_1,\dots, v_n$ labelled in a clockwise order. We define the \textit{trip permutation} $\pi_{\mathcal G}$ as follows. We start at a boundary vertex $v_i$ and form a path along the edges of $\mathcal G$ by turning maximally right at an internal black vertex and maximally left at an internal white vertex. We end up at a boundary vertex $v_{\pi(i)}$ and define $\pi_{\mathcal G}=[\pi(1),\dots,\pi(n)]\in S_n$.
\end{definition}

It is a fact that plabic graphs in one equivalence class have the same trip permutation. Further, it was proven by Postnikov in \cite[Theorem 13.4]{Pos06} that plabic graphs with the same trip permutation are connected by moves (M1)-(M3) and are therefore equivalent.
Let $\pi_{n-k,n}=(k+1,k+2,\dots,n,1,2,\dots,k)$. From now on we focus on plabic graphs $\mathcal G$ with trip permutation $\pi_{\mathcal G}=\pi_{n-k,n}$. Each path $v_i$ to $v_{\pi_{n-k,n}(i)}$ defined above, divides the disk into two regions. 
We label every face in the region to the left of the path by $i$. After repeating this for every $1\le i\le n$, all faces have a labelling by an $k$-element subset of $[n]$. 
Every such $k$-element subset defines a Young diagram that fits into an $(n-k)\times k$-rectangle of boxes.
We denote by $\mathcal{P_G}$ the set of all such subsets (resp. their associated Young diagrams) for a fixed plabic graph $\mathcal G$.
The cardinality of $\mathcal P_{\G}$ is $d+1$.

A face of a plabic graph is called \emph{internal}, if it does not intersect with the boundary of the disk. Other faces are called \emph{boundary faces}.
Following \cite{RW17} we define an orientation on a plabic graph. 
This is the first step in establishing the \emph{flow model} introduced by Postnikov, which we use to define plabic degrees on the Pl\"ucker coordinates.

\begin{definition}\label{def: perfect orientation}
An orientation $\mathcal O$ of a plabic graph $\mathcal G$ is called \emph{perfect}, if every internal white vertex has exactly one incoming arrow and every internal black vertex has exactly one outgoing arrow. 
The set of boundary vertices that are sources is called the \emph{source set} and is denoted by $I_{\mathcal O}$.
\end{definition}

Postnikov showed in \cite{Pos06} that every reduced plabic graph with trip permutation $\pi_{n-k,n}$ has a perfect orientation with source set of order $k$. See Figure \ref{fig:ExamplePerfectOrientation} for a plabic graph with trip permutation $\pi_{3,5}$.

Index the standard basis of $\mathbb Z^{\mathcal P_\G}=\mathbb Z^{d+1}$ by the faces of the plabic graph $\mathcal G$, where $d=k(n-k)$.
Given a perfect orientation $\mathcal O$ on $\mathcal G$, every directed path $\p$ from a boundary vertex in the source set to a boundary vertex that is a sink, divides the disk in two parts. 
The \emph{weight} $\wei(\p)\in\mathbb Z_{\ge 0}^{\mathcal P_\G}$ has entry $1$ in the position corresponding to a face $F$ of $\mathcal G$, if $F$ is to the left of $\p$ with respect to the orientation.
The \emph{degree} $\deg_{\mathcal G}(\p)\in \mathbb Z_{\ge 0}$ is defined the number of internal faces to the left of the path. 
The boundary face between the boundary vertices $k$ and $k+1$ never lies to the left of any path and therefore is also referred to as $F_{\varnothing}$. 

\begin{definition}\label{def:flow}
For a set of boundary vertices $J$ with $\vert J\vert =\vert I_{\mathcal O}\vert$, we define a \emph{$J$-flow} as a collection of self-avoiding, vertex disjoint directed paths with sources $I_{\mathcal O}-(J\cap I_{\mathcal O})$ and sinks $J-(J\cap I_{\mathcal O})$. 
Let $I_{\mathcal O}-(J\cap I_{\mathcal O})=\{j_1,\dots,j_r\}$ and $\bfff=\{\p_{j_1},\dots,\p_{j_r}\}$ be a flow where each path $\p_{j_i}$ has sink $j_i$.
Then the \emph{weight of the flow} is $\wei(\bfff):=\wei(\p_{j_1})+\dots+\wei(\p_{j_r})$.
Similarly, we define the \emph{degree of the flow} as $\deg_\G(\bfff)=\deg_\G(\p_{j_1})+\dots +\deg_\G(\p_{j_r})$.
By $\mathcal{F}_J$ we denote the set of all $J$-flows in $\mathcal G$ with respect to $\mathcal O$.
\end{definition}

\subsection{Valuation and plabic degree}\label{sec:appendix val}

In \cite{RW17} Rietsch-Williams use the cluster structure on $\Gr_k(\mathbb C^n)$ (due to Scott, see \cite{Sco06}) to define a valuation on $\mathbb C(\Gr_k(\mathbb C^n))\setminus \{0\}$ for every seed.
In fact, a plabic graph $\G$ defines a seed in the corresponding cluster algebra. 
A combinatorial algorithm associates a quiver with $\G$ (see e.g. \cite[Definition~3.8]{RW17}).
The corresponding cluster is a set of Pl\"ucker coordinates $\bar p_J$, where $J$ is a face label in $\G$ as described above.

Let $A_{k,n}:=\mathbb C[\Gr(k,n)]$. 
We recall the definition of the valuation $\val_\G$ from \cite[\S8]{RW17}. 
By \cite[\S6]{Pos06} for every plabic graph $\G$ there exists a map $\Phi_\G:(\mathbb C^*)^{\mathcal P_\G}\to \Gr_k(\mathbb C^n)$ sending $(x_\mu)_{\mu\in\mathcal P_\G}$ to $A\in \mathbb C^{n\times k}$, where
\[
A_{ij}=(-1)^{\#\{ i'\in [k]\vert i<i'<j \}} \sum_{\text{path } \p:i\to j} x^{\wei(\p)} \in \mathbb C[x_\mu]_{\mu\in \mathcal P_\G}.
\]
Here $x^{\wei(\p)}$ denotes the monomial with exponent vector $\wei(\p)\in\mathbb \{0,1\}^{\mathcal P_\G}$.
Fix an order on the coordinates $\{x_\mu\}_{\mu\in \mathcal P_\G}$ and let $\prec$ be the total order on $\mathbb Z^{\mathcal P_\G}$ be the corresponding lexicographic order (see \cite[Definition~8.1]{RW17}).
The pullback of $\Phi_\G$ satisfies $\Phi_\G^*(\bar p_J)\in \mathbb C[x_\mu^{\pm 1}\vert \mu \in \mathcal P_\G]$ for $J\in\binom{[n]}{k}$.
Moreover, every polynomial in Pl\"ucker coordinates has a unique Laurent polynomial expression in $\{x_\mu\}_{\mu\in \mathcal P_\G}$.
Suppose $f\in A_{k,n}$ has expression $\sum_{i=1}^s a_ix^{m_i}$ for $a_i\in \mathbb C$ and $m_i\in \mathbb Z^{\mathcal P_\G}$.
Then the lowest term valuation $\val_\G:\mathbb C(\Gr_k(\mathbb C^n)) \setminus \{0\} \to (\mathbb Z^{\mathcal P_\G},\prec)$ is defined as
\[
\val_\G(f):=\min{}_{\text{lex}} \{ m_i \mid a_i\not =0\}.
\]
For a rational function $h=\frac{f}{g}$ it is defined as $\val_\G(h):=\val_\G(f)-\val_\G(g)$.
For Pl\"ucker coordinates the images of $\val_\G$ can be computed explicitly using the combinatorics of the plabic graph.
For $J\in\binom{[n]}{k}$ let $\bfff_J\in\ff_J$ be the flow with $\deg_\G(\bfff_J)=\min_{\prec}\{\deg_\G(\bfff)\mid \bfff\in\ff_J\}$.
Then on a Pl\"ucker coordinate $\bar p_J\in A_{k,n}$ the valuation $\val_\G$ is given by
\begin{eqnarray}\label{eq:RW val}
\val_\G(\bar p_J)=\wei(\bfff_J) \in \mathbb Z^{\mathcal P_\G}.
\end{eqnarray}
In fact, $\mathcal P_\G$ contains one label $\mu$ whose corresponding face never contributes to the weight of any flow. 
With our convention it is the boundary face between the vertices $k$ and $k+1$, call it $F_\varnothing$.
Therefore, we can omit the corresponding variable and have $\val_\G:A_{k,n}\setminus\{0\}\to \mathbb Z^{\mathcal P_\G-\varnothing}\cong \mathbb Z^d$.
In particular, this implies that the rank of $\val_G$ is $d=\dim(\Gr_n(\mathbb C^n))$ which is one less than the Krull-dimension of $A_{k,n}$.
In order to have a valuation that fits into our framework, we slightly modify $\val_\G$ and define
\begin{eqnarray}\label{eq:valhatG}
\hat\val_\G: A_{k,n}\setminus\{0\}\to \mathbb Z^{d+1} \quad \text{ given by } \quad \hat\val_\G(f)=(\deg f, \val_\G(f)).
\end{eqnarray}
Note that $\hat \val_\G$ is a full-rank valuation on $A_{k,n}$. 
Further, $M_{\hat\val_\G}\in \mathbb Z^{(d+1)\times \binom{n}{k}}$ differs from $M_{\G}\in \mathbb Z^{d\times \binom{n}{k}}$ by the row $(1,\dots,1) \in \mathbb Z^{\binom{n}{k}}$.
As $I_{k,n}$ is homogeneous, seen as a weight vector, we have $(1,\dots,1)\in L_{I_{k,n}}$. 
In particular,
\begin{eqnarray}\label{eq:init same}
\init_{M_\G}(I_{k,n})=\init_{M_{\hat\val_\G}}(I_{k,n}).
\end{eqnarray}

In \cite{BFFHL} they define closely related to the valuation the following notion of degree for Pl\"ucker variables in $\mathbb C[p_J]_{J}$ and associate a weight vector in $\mathbb Z^{\binom{n}{k}}$.
For an example, consider Example~\ref{exp:plabic val} above. 

\begin{definition}\label{def: plabic deg}
For $J\in\binom{[n]}{k}$ and a plabic graph $\mathcal{G}$, the \emph{plabic degree} of the Pl\"ucker variable $p_J$ is defined as
\[
\deg_{\mathcal G}(p_J):=\min\{\deg_{\mathcal G}(\bfff)\mid \bfff\in \mathcal F_J\} \in \mathbb Z_{\ge 0}.
\]
It gives rise to the \emph{plabic weight vector} ${\bf w}_{\mathcal G}\in\mathbb Z^{\binom{n}{k}}$ defined by $({\bf w}_{\mathcal G})_J:=\deg_{\mathcal G}(p_J)$.
\end{definition}

By \cite[Lemma~3.2]{PSW09} and its proof, the plabic degree is independent of the choice of the perfect orientation. We therefore fix the perfect orientation by choosing the source set $I_{\mathcal O}=[k]$. 
The following proposition guarantees that the degree (and the valuation) are well-defined.
It is a reformulation of the original statement adapted to our notion degree.

\begin{proposition*}(\cite[Corollary~12.4]{RW17})\label{prop: unique min flow}
There is a unique $J$-flow in $\mathcal G$ with respect to $\mathcal O$ with degree equal to $\deg_{\mathcal{G}}( p_J)$.
\end{proposition*}

\subsection{\cite{GHKK14}'s ${\bf g}$-vector valuation vs. \cite{RW17}'s valuation}\label{sec:RW and GHKK}

Gross, Hacking, Keel and Kontsevich construct vector space bases for cluster algebras in \cite{GHKK14}.
Among other powerful applications their so-called \emph{theta basis} can be used to construct toric degenerations for partially compactified cluster varieties.

For the Grassmannian their toric degeneration can be formulated in terms of valuations and Newton-Okounkov bodies. 
They rely on Fomin and Zelevinsky's principal coefficients (introduced in \cite{FZ07}) and the associated multiweights for cluster monomials called \emph{{\bf g}-vectors}. 
Generalized ${\bf g}$-vectors are introduced in \cite[Definition 5.10]{GHKK14} for elements of the theta basis.
For any seed $s$ the assignment of its ${\bf g}$-vector to a theta basis element extends to a full-rank valuation ${\bf g}_{s}: A_{k,n}\setminus \{0\} \to \mathbb Z^{d}$.
For us the following straight forward lemma is important.

\begin{lemma}\label{lem:gv identity}
For a plabic graph $\mathcal G$ consider the set of Pl\"ucker coordinated $\bar p_J$, where $J\in \mathcal P_{\G}-\varnothing$.
Then the matrix with rows ${\bf g}_{\mathcal G}(\bar p_{J})$ is (up to permutation of the rows) the identity matrix.
\end{lemma}

The following result is a direct consequence of the main theorem in \cite{BCMN}:

\begin{theorem*}[\cite{BCMN}]
For every plabic graph $\mathcal G$ there exists a linear map $\bar p^*_{\mathcal G}:\mathbb R^d\to \mathbb R^d$ inducing a unimodular equivalence between
\[
\Delta(A_{k,n},\val_{\mathcal G})\cong \Delta(A_{k,n},{\bf g}_{\mathcal G}).
\]
\end{theorem*}

\begin{example}
We consider the plabic graph $\G^{\rm rec}$ for $\Gr_2(\mathbb C^5)$ as in Figure~\ref{fig:ExamplePerfectOrientation}.
For a lattice $N\cong \mathbb Z^7$ we fix an ordered basis $\{e_{35},e_{25},e_{45},e_{12},e_{23},e_{34}\}$ corresponding to the faces of $\G^{\rm rec}$.
Let $M:=N^*$ be the dual lattice with dual basis (in order) $\{f_{35},f_{25},f_{45},f_{12},f_{23},f_{34}\}$. 
Then $\bar p^*_{\G^{\rm rec}}:N\otimes_{\mathbb Z}\mathbb R\to M\otimes_{\mathbb Z}\mathbb R$ with respect to our chosen bases is given by the matrix
\[
\left(
    \begin{smallmatrix}
    0 & 1 & -1 & 0 & 0 & -1 & 1 \\
    -1 & 0 & 0 & 1 & -1 & 1 & 0 \\
    1 & 0 & 1 & -1 & 0 & 0 & -1 \\
    0 & -1 & 0 & 1 & 0 & 0 & 0 \\
    0 & 1 & 0 & -1 & 1 & -1 & 0 \\
    1 & -1 & 0 & 0 & 0 & 1 & -1 \\
    -1 & 0 & 0 & 0 & 0 & 0 & 1
    \end{smallmatrix}
\right)
\]
The Newton--Okounkov cone $C(A_{2,5},\val_{\G^{\rm rec}})$ is cut out by the tropicalized Marsh-Rietsch superpotential $W=W_1+\dots +W_5$ (defined in \cite{MR_B-model}). 
Expressed in the Pl\"ucker coordinates corresponding to $\G^{\rm rec}$ it is of form
\[
W_1=\frac{\bar p_{34}}{\bar p_{35}} + \frac{\bar p_{23}\bar p_{45}}{\bar p_{35}\bar p_{25}} + \frac{\bar p_{12}\bar p_{45}}{\bar p_{15}\bar p_{25}}, W_2= q\frac{\bar p_{25}}{\bar p_{12}}, W_3= \frac{\bar p_{15}}{\bar p_{25}} + \frac{\bar p_{12}\bar p_{35}}{\bar p_{23}\bar p_{25}}, W_4= \frac{\bar p_{25}}{\bar p_{35}} + \frac{\bar p_{23}\bar p_{45}}{\bar p_{34}\bar p_{35}}, W_5= \frac{\bar p_{35}}{\bar p_{45}}.
\]
The convention to make sense of $\bar p^*$ as written above is $\bar p_{ij}:=z^{f_{ij}}$.
Note that the image of $\val_{\G^{\rm rec}}$ lies inside a hyperplane defined by $e_{45}=0$.
By \cite{RW17} the Newton--Okounkov polytope associated to $\val_{\G^{\rm rec}}:A_{2,5}\setminus\{0\}\to N$ is
\[
\Delta(A_{2,5},\val_{\G^{\rm rec}})= \{W_1^{\trop}\ge 0\}\cap \{W_2^{\trop}\ge -1\} \cap \{W_3^{\trop}\ge 0\} \cap \{W_4^{\trop}\ge 0\} \cap \{W_5^{\trop}\ge 0\}.
\]
For example, $W_5^{\trop}=f_{35}-f_{45}$ is the normal vector for the inequality $e_{35}-e_{45}\ge 0$. 
The lattice points of $\Delta(A_{2,5},\val_{\G^{\rm rec}})$ can be found on the left in Table~\ref{tab:val 2-5}.
The tropicalized Gross-Hacking-Keel-Kontsevich potential $W_{\rm GHKK}=\vartheta_1+\dots+\vartheta_5$ cuts out the Newton--Okounkov cone $C(A_{2,5},{\bf g}_{\G^{\rm rec}})$. 
Expressed in our seed we have
\[
\vartheta_1 = z^{-e_{15}}(1 + z^{-e_{25}}(1 + z^{-e_{35}})), \vartheta_2 = z^{-e_{12}}, \vartheta_3= z^{-e_{23}}(1 + z^{-e_{25}}), \vartheta_4 = z^{-e_{34}}(1 + z^{-e_{25}}), \vartheta_5 = z^{-e_{45}}.
\]
The image of $\Delta(A_{2,5},\val_{\G^{\rm rec}})$ under $\bar p^*_{\G^{\rm rec}}$ is
\[
P= \{\vartheta_1^{\trop}\ge 0\}\cap \{\vartheta_2^{\trop}\ge -1\} \cap \{\vartheta_3^{\trop}\ge 0\} \cap \{\vartheta_4^{\trop}\ge 0\} \cap \{\vartheta_5^{\trop}\ge 0\}.
\]
It is the Newton--Okounkov polytope for a linearly equivalent divisor. 
The Newton--Okounkov polytope $\Delta(A_{2,5},{\bf g}_{\G^{\rm rec}})$ associated to ${\bf g}_{\G^{\rm rec}}:A_{2,5}\setminus \{0\}\to M$ equals $\{W_{\rm GHKK}^{\trop}\le 0\} \cap \{\sum f_{ij}=1\}$.
It is unimodularly equivalent to $P$. 
The polytope $\Delta(A_{2,5},{\bf g}_{\G^{\rm rec}})$ is sent to $P$ by the a linear map $q_{\G^{\rm rec}}:M\otimes_{\mathbb Z}\mathbb R\to M\otimes_{\mathbb Z}\mathbb R$ defined by
\[
\left(\begin{smallmatrix}
1 & 0 & 2 & 0 & 0 & 0 & 0 \\
0 & 1 &-1 & 0 & 0 & 0 & 0 \\
-2&-2 &-2 &-1 & 0 &-1 &-1 \\
0 & 1 & 0 & 0 & 0 & 1 & 0 \\
-1&-1 &-1 & 0 & 0 & 0 & 1 \\
-1& 0 &-1 & 0 & 0 & 0 & 0 \\
1 & 1 & 1 & 1 & 0 & 0 & 0 \\
\end{smallmatrix}\right).
\]

\end{example}
\end{appendices}

\footnotesize


\begin{thebibliography}{BLMM17}

\bibitem[AB04]{AB04}
Valery Alexeev and Michel Brion.
\newblock Toric degenerations of spherical varieties.
\newblock {\em Selecta Math. (N.S.)}, 10(4):453--478, 2004.

\bibitem[And13]{An13}
Dave Anderson.
\newblock Okounkov bodies and toric degenerations.
\newblock {\em Math. Ann.}, 356(3):1183--1202, 2013.

\bibitem[Bay82]{Ba82}
David~A. Bayer.
\newblock {\em The division algorithm and the {H}ilbert scheme}.
\newblock ProQuest LLC, Ann Arbor, MI, 1982.
\newblock Thesis (Ph.D.)--Harvard University.

\bibitem[BCMN19]{BCMN}
L.~Bossinger, M.-W. Cheung, T.~Magee, and A.~{N\'ajera Ch\'avez}.
\newblock On cluster duality for {G}rassmannians.
\newblock {\em In preparation}, 2019.

\bibitem[BFF{\etalchar{+}}18]{BFFHL}
Lara Bossinger, Xin Fang, Ghislain Fourier, Milena Hering, and Martina Lanini.
\newblock Toric degenerations of {${\rm Gr}(2, n)$} and {${\rm Gr}(3, 6)$} via
  plabic graphs.
\newblock {\em Ann. Comb.}, 22(3):491--512, 2018.

\bibitem[BG09]{BG09}
Winfried Bruns and Joseph Gubeladze.
\newblock {\em Polytopes, rings, and {$K$}-theory}.
\newblock Springer Monographs in Mathematics. Springer, Dordrecht, 2009.

\bibitem[BLMM17]{BLMM}
Lara Bossinger, Sara Lamboglia, Kalina Mincheva, and Fatemeh Mohammadi.
\newblock Computing toric degenerations of flag varieties.
\newblock In {\em Combinatorial algebraic geometry}, volume~80 of {\em Fields
  Inst. Commun.}, pages 247--281. Fields Inst. Res. Math. Sci., Toronto, ON,
  2017.

\bibitem[Bos18]{Thesis}
Lara Bossinger.
\newblock {\em Toric degenerations: a bridge between representation theory,
  tropical geometry and cluster algebras}.
\newblock PhD thesis, Universit{\"a}t zu K{\"o}ln, 2018.

\bibitem[Bos21]{B-birat}
Lara Bossinger.
\newblock Birational sequences and the tropical {G}rassmannians.
\newblock {\em J. Algebra}, doi: 10.1016/j.jalgebra.2021.04.028, 2021.

\bibitem[Bri04]{Bri04}
Michel Brion.
\newblock Lectures on the geometry of flag varieties.
\newblock {\em arXiv preprint arXiv:0410240}, 2004.

\bibitem[BZ01]{BZ01}
Arkady Berenstein and Andrei Zelevinsky.
\newblock Tensor product multiplicities, canonical bases and totally positive
  varieties.
\newblock {\em Inventiones mathematicae}, 143(1):77--128, 2001.

\bibitem[Cal02]{Cal02}
Philippe Caldero.
\newblock Toric degenerations of {S}chubert varieties.
\newblock {\em Transform. Groups}, 7(1):51--60, 2002.

\bibitem[CLS11]{CLS11}
David~A. Cox, John~B. Little, and Henry~K. Schenck.
\newblock {\em Toric varieties}.
\newblock American Mathematical Soc., 2011.

\bibitem[Eis95]{Eis13}
David Eisenbud.
\newblock {\em Commutative algebra}, volume 150 of {\em Graduate Texts in
  Mathematics}.
\newblock Springer-Verlag, New York, 1995.
\newblock With a view toward algebraic geometry.

\bibitem[FFL17]{FFL15}
Xin Fang, Ghislain Fourier, and Peter Littelmann.
\newblock Essential bases and toric degenerations arising from birational
  sequences.
\newblock {\em Adv. Math.}, 312:107--149, 2017.

\bibitem[FZ02]{FZ02}
Sergey Fomin and Andrei Zelevinsky.
\newblock Cluster algebras. {I}. {F}oundations.
\newblock {\em J. Amer. Math. Soc.}, 15(2):497--529, 2002.

\bibitem[FZ07]{FZ07}
Sergey Fomin and Andrei Zelevinsky.
\newblock Cluster algebras. {IV}. {C}oefficients.
\newblock {\em Compos. Math.}, 143(1):112--164, 2007.

\bibitem[GHKK18]{GHKK14}
Mark Gross, Paul Hacking, Sean Keel, and Maxim Kontsevich.
\newblock Canonical bases for cluster algebras.
\newblock {\em J. Amer. Math. Soc.}, 31(2):497--608, 2018.

\bibitem[IW18]{IW18}
Nathan Ilten and Milena Wrobel.
\newblock Khovanskii-finite valuations, rational curves, and torus actions.
\newblock {\em arXiv preprint arXiv:1807.08780}, 2018.

\bibitem[Kav15]{Kav15}
Kiumars Kaveh.
\newblock Crystal bases and {N}ewton-{O}kounkov bodies.
\newblock {\em Duke Math. J.}, 164(13):2461--2506, 2015.

\bibitem[KK12]{KK12}
Kiumars Kaveh and Askold~G. Khovanskii.
\newblock Newton-{O}kounkov bodies, semigroups of integral points, graded
  algebras and intersection theory.
\newblock {\em Ann. of Math. (2)}, 176(2):925--978, 2012.

\bibitem[KM19]{KM16}
Kiumars Kaveh and Christopher Manon.
\newblock Khovanskii bases, higher rank valuations, and tropical geometry.
\newblock {\em SIAM J. Appl. Algebra Geom.}, 3(2):292--336, 2019.

\bibitem[LB09]{LB09}
Venkatramani Lakshmibai and Justin Brown.
\newblock {\em Flag varieties}, volume~53 of {\em Texts and Readings in
  Mathematics}.
\newblock Hindustan Book Agency, New Delhi, 2009.
\newblock An interplay of geometry, combinatorics, and representation theory.

\bibitem[Lit98]{Lit98}
P.~Littelmann.
\newblock Cones, crystals, and patterns.
\newblock {\em Transform. Groups}, 3(2):145--179, 1998.

\bibitem[LM09]{LM09}
Robert Lazarsfeld and Mircea Musta\c{t}\u{a}.
\newblock Convex bodies associated to linear series.
\newblock {\em Ann. Sci. \'Ec. Norm. Sup\'er. (4)}, 42(5):783--835, 2009.

\bibitem[MR13]{MR_B-model}
Robert {Marsh} and Konstanze {Rietsch}.
\newblock The {B}-model connection and mirror symmetry for {G}rassmannians.
\newblock {\em arXiv preprint arXiv:1307.1085}, 2013.

\bibitem[MS05]{MS05}
Ezra Miller and Bernd Sturmfels.
\newblock {\em Combinatorial commutative algebra}, volume 227 of {\em Graduate
  Texts in Mathematics}.
\newblock Springer-Verlag, New York, 2005.

\bibitem[MS15]{M-S}
Diane Maclagan and Bernd Sturmfels.
\newblock {\em Introduction to tropical geometry}, volume 161 of {\em Graduate
  Studies in Mathematics}.
\newblock American Mathematical Society, Providence, RI, 2015.

\bibitem[MS18]{MoSh}
Fatemeh Mohammadi and Kristin Shaw.
\newblock Toric degenerations of {G}rassmannians from matching fields.
\newblock {\em Algebr. Comb.} 2, 2019.


\bibitem[Pos06]{Pos06}
Alexander Postnikov.
\newblock Total positivity, {G}rassmannians, and networks.
\newblock {\em arXiv preprint arXiv:math/0609764}, 2006.

\bibitem[PSW09]{PSW09}
Alexander Postnikov, David Speyer, and Lauren Williams.
\newblock Matching polytopes, toric geometry, and the totally non-negative
  {G}rassmannian.
\newblock {\em J. Algebraic Combin.}, 30(2):173--191, 2009.

\bibitem[RW19]{RW17}
K.~Rietsch and L.~Williams.
\newblock Newton--{O}kounkov bodies, cluster duality, and mirror symmetry for
  {G}rassmannians.
\newblock {\em Duke Math. J.}, 168(18):3437--3527, 2019.

\bibitem[Sco06]{Sco06}
Joshua~S. Scott.
\newblock Grassmannians and cluster algebras.
\newblock {\em Proc. London Math. Soc. (3)}, 92(2):345--380, 2006.

\bibitem[SS04]{SS04}
David Speyer and Bernd Sturmfels.
\newblock The tropical {G}rassmannian.
\newblock {\em Adv. Geom.}, 4(3):389--411, 2004.

\bibitem[Tei03]{T03}
Bernard Teissier.
\newblock Valuations, deformations, and toric geometry.
\newblock In {\em Valuation theory and its applications, {V}ol. {II}
  ({S}askatoon, {SK}, 1999)}, volume~33 of {\em Fields Inst. Commun.}, pages
  361--459. Amer. Math. Soc., Providence, RI, 2003.

\bibitem[VP89]{PV89}
\`Ernest~B. Vinberg and Vladimir~L. Popov.
\newblock Invariant theory.
\newblock In {\em Algebraic geometry, 4 ({R}ussian)}, Itogi Nauki i Tekhniki,
  pages 137--314, 315. Akad. Nauk SSSR, Vsesoyuz. Inst. Nauchn. i Tekhn.
  Inform., Moscow, 1989.

\end{thebibliography}

\newcommand{\etalchar}[1]{$^{#1}$}

\normalsize
\bigskip

\textsc{Instituto de Matem\'aticas UNAM Unidad Oaxaca, 
Antonio de Le\'on 2, altos, Col. Centro, 
Oaxaca de Ju\'arez, CP. 68000,
Oaxaca, M\'exico}

\emph{E-mail address:} \texttt{lara@im.unam.mx}

\end{document}